\documentclass[10pt]{amsart}
\usepackage[margin=1in]{geometry}
\usepackage{amsmath,amsthm,amssymb,amsfonts}
\usepackage{subcaption} 
\usepackage[countmax]{subfloat}
\usepackage[pdfencoding=auto]{hyperref} 
\usepackage{xcolor}
\usepackage{float}

\usepackage{comment}

\usepackage{graphicx}
\graphicspath{{images/}}

\usepackage{setspace} 

\usepackage{enumitem}

\newcommand{\Z}{\mathbb{Z}}
\newcommand{\R}{\mathbb{R}}
\newcommand{\CP}{\mathbb{CP}^2}
\newcommand{\CPbar}{\overline{\mathbb{CP}}^2}
\newcommand{\nCP}{\# n \mathbb{CP}^2}
\newcommand{\nCPbar}{\# n \overline{\mathbb{CP}}^2}
\newcommand{\rCP}{\# r \mathbb{CP}^2}
\newcommand{\rCPbar}{\# r \overline{\mathbb{CP}}^2}
\newcommand{\CPone}{\mathbb{CP}^1}
\newcommand{\CPonebar}{\overline{\mathbb{CP}}^1}
\def\inter{\mathop{\rm int}}
\def\rnk{\mathop{\rm rank}}

\newcommand{\PF}{\mathbb{PF}}
\newcommand{\zsg}{\phi : S^3_0(K) \rightarrow S^3_0(K')}
\newcommand{\MPfam}{L(a,b,c,d,e,f)}

\usepackage[backend=biber,style=alphabetic,citestyle=alphabetic,maxnames=10,maxalphanames=99]{biblatex}

\renewbibmacro{in:}{%
  \ifentrytype{article}
    {}
    {\bibstring{in}%
     \printunit{\intitlepunct}}}

\usepackage{hyperref}

\newtheorem{thm}{Theorem}[section]
\newtheorem{lem}[thm]{Lemma}
\newtheorem{cor}[thm]{Corollary}
\newtheorem{mydef}[thm]{Definition}
\newtheorem{prop}[thm]{Proposition}
\newtheorem{remark}[thm]{Remark}
\newtheorem{conj}[thm]{Conjecture}

\newtheorem{prob}[thm]{Problem}

\newtheorem*{refthm}{Theorem}
\newtheorem*{SPC4}{Smooth 4-Dimensional Poincar\'e Conjecture (SPC4)}

\title {Trace Embeddings from Zero Surgery Homeomorphisms}
\author{Kai Nakamura}
\address{  University of Texas at Austin\\
      Department of Mathematics\\
      2515 Speedway. Austin, TX 78712
   USA}
\email{kainakamura@utexas.edu}
\date{\today}

\begin{document}

\begin{abstract}
Manolescu and Piccirillo recently initiated a program to construct an exotic $S^4$ or $\nCP$ by using zero surgery homeomorphisms and Rasmussen's $s$-invariant \cite{zero_surg_exotic}.
They find five knots that if any were slice, one could construct an exotic $S^4$ and disprove the Smooth $4$-dimensional Poincar\'e conjecture.
We rule out this exciting possibility and show that these knots are not slice.
To do this, we use a zero surgery homeomorphism to relate slice properties of two knots \textit{stably} after a connected sum with some $4$-manifold.
Furthermore, we show that our techniques will extend to the entire infinite family of zero surgery homeomorphisms constructed by Manolescu and Piccirillo.
However, our methods do not completely rule out the possibility of constructing an exotic $S^4$ or $\nCP$ as Manolescu and Piccirillo proposed.
We explain the limits of these methods hoping this will inform and invite new attempts to construct an exotic $S^4$ or $\nCP$.
We also show a family of homotopy spheres constructed by Manolescu and Piccirillo using annulus twists of a ribbon knot are all standard.
\end{abstract}

\maketitle

\section{Introduction}
\subsection{Background}
The study of $4$-manifolds is distinguished by the remarkable difference between smooth and topological $4$-manifolds compared to other dimensions.
This manifests in the failure of Smale's h-cobordism theorem which Smale used to prove the high dimensional Poincar\'e conjecture \cite{hcobord}.
This left open the $4$-dimensional Poincar\'e conjecture which asserts that every smooth $4$-manifold homotopy equivalent to $S^4$, i.e. a homotopy $S^4$, is homeomorphic to $S^4$.
Freedman resolved the $4$-dimensional Poincar\'e conjecture and showed that simply connected, smooth $4$-manifolds are determined up to homeomorphism by their intersection forms \cite{Freedman}.
Donaldson contrasted this topological simplicity with his Diagonalization Theorem and showed that $4$-manifolds do not smoothly admit such a straightforward classification \cite{Donaldson_Diag}.
The resulting study of $4$-manifolds have yielded an abundance of exotic pairs of $4$-manifolds: pairs of $4$-manifolds homeomorphic, but not diffeomorphic to each other.
Unique to dimension $4$ are phenomena such as infinite families of exotic smooth structures on $\R^4$ and small closed $4$-manifolds such as $\CP \# 2 \CPbar$ \cite{uncount_ex_R4,exotic_small_closed}.
Remarkably, this exotic behavior has not been shown to occur with the $4$-sphere: the most basic example of a closed $4$-manifold.
This is the focus of the Smooth $4$-dimensional Poincar\'e conjecture (SPC4).
\begin{SPC4}
\label{SPC4}
Every homotopy $4$-sphere is diffeomorphic to the standard $4$-sphere
\end{SPC4}
\noindent The various flavors of the Poincar\'e conjecture motivated and revolutionized 20th century topology with SPC4 the last low dimensional case that remains unresolved.

Historically, the consensus among experts is that SPC4 is false due to the aforementioned exotica and the many constructions of homotopy spheres that are not known to be standard.
The difficulty with exhibiting an exotic $S^4$ is that the invariants used to distinguish smooth structures are typically known to vanish on homotopy $4$-spheres.
This changed with Rasmussen's invention of his eponymous $s$-invariant \cite{Khov_slice}, a slice obstruction coming from Khovanov Homology \cite{khov}.
A knot is smoothly slice if it is the boundary of a smooth properly embedded disk in $B^4$.
Rasmussen defined the $s$-invariant $s(K) \in 2\Z$ for any knot $K$ and showed that if $s(K) \neq 0$, then $K$ is not slice.
Unlike prior slice obstructions, it is not clear that the $s$-invariant vanishes on knots slice in a homotopy ball other than $B^4$.
In theory, one could show that a homotopy sphere $\Sigma$ is exotic by finding a knot $K$ slice in the homotopy ball $\Sigma - \inter(B^4)$ and has $s(K) \neq 0$.
Then $K$ would not be slice in the standard $B^4$ and sliceness of $K$ smoothly distinguishes $\Sigma$ from the standard $S^4$.

Freedman, Gompf, Morrison, and Walker (FGMW) attempted this strategy on the intensely studied family of Cappell-Shaneson spheres $\Sigma_n$.
They found knots slice in $\Sigma_n -\inter(B^4)$, hoping that one of these knots would have non-vanishing $s$-invariant.
They were only able to do the calculations for two of their knots and got zero for both \cite{Man_and_Machine}.
Surprisingly, only six days after FGMW posted their results, Akbulut showed that all $\Sigma_n$ are standard \cite{CS_standard}.
Indirectly, this shows that all of the knots FGMW considered had vanishing $s$-invariant.

Piccirillo's acclaimed proof that the Conway knot is not slice has renewed interest in the $s$-invariant.
Piccirillo's proof takes advantage of and makes apparent the uniqueness of Rasmussen's $s$-invariant among other slice obstructions \cite{Conway_knot}.
It now seems more likely that the $s$-invariant could be used to distinguish a homotopy sphere from $S^4$.
However, the Cappell-Shaneson spheres $\Sigma_n$ were the most promising potentially exotic homotopy spheres.
With $\Sigma_n$ now standardized, we are left with a dearth of homotopy spheres.

Recently, Piccirillo worked with Manolescu to revive this idea of FGMW to use the $s$-invariant to find an exotic $S^4$.
Unlike FGMW, they don't use knots slice in a known homotopy sphere and in some sense, they reverse the FGMW approach.
They propose a way to build a new homotopy sphere $\Sigma$ which comes with such a knot already.
To construct an exotic $S^4$, they propose to take a pair of knots $(K,K')$ which satisfy three conditions:
\[K \text{ is slice},   \ \ \ s(K') \neq 0, \ \ \ \zsg\]
This would allow one to construct a homotopy sphere $\Sigma$ where $K'$ is slice in $\Sigma - \inter(B^4)$.
Then $\Sigma$ would have the properties that FGMW wanted and would be an exotic $S^4$.
Manolescu and Piccirillo initiated a search for such a pair of knots, but did not find any that satisfied all three conditions.
They did find pairs $(K,K')$ which have homeomorphic zero surgeries, $s(K') < 0$, and could not immediately determine the sliceness of $K$.
They conclude the following:
\begin{figure}
\captionsetup[subfigure]{font=normalsize, labelformat=empty}

\centering
\subfloat[$K_1$]{\includegraphics[scale=.15]{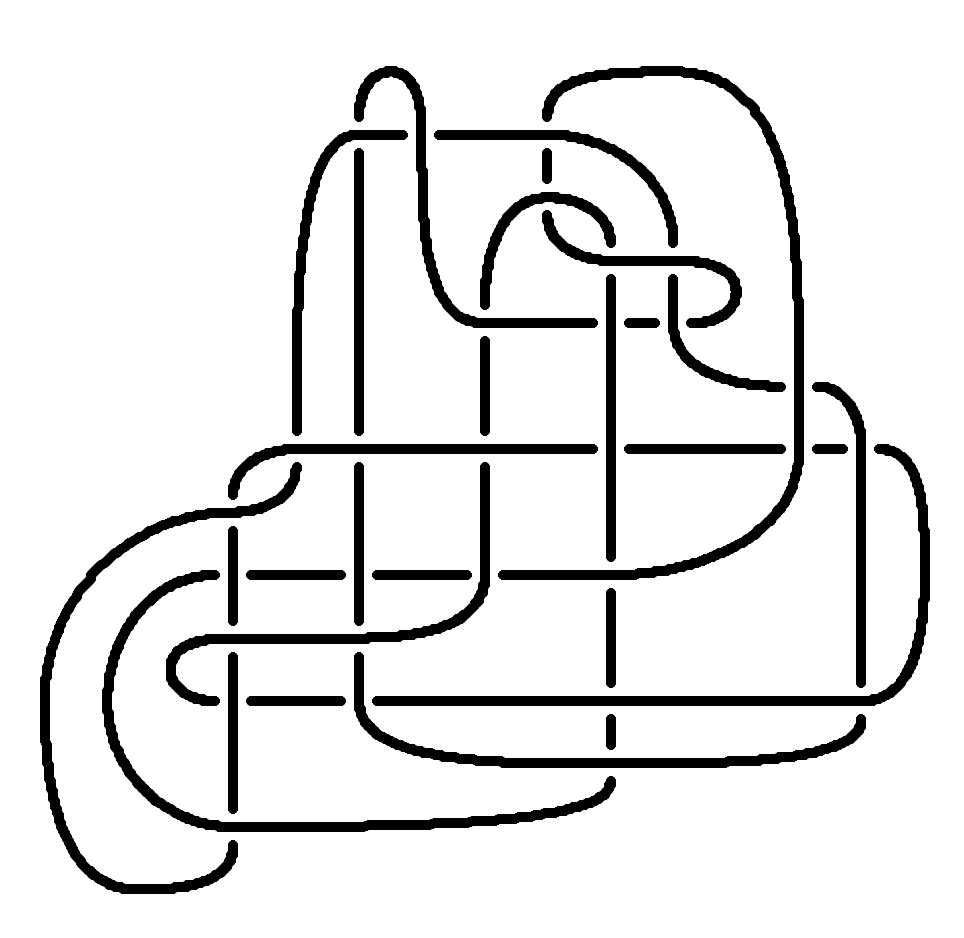}} \ \ \
\subfloat[$K_2$]{\includegraphics[scale=.08]{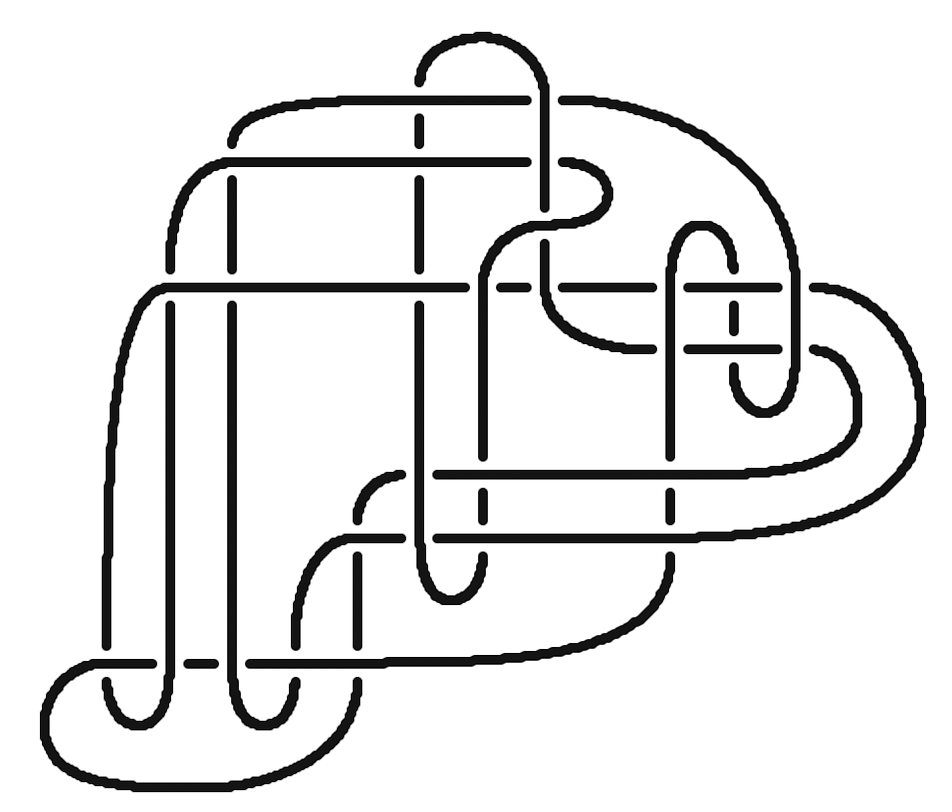}} \ \ \
\subfloat[$K_3$]{\includegraphics[scale=.08]{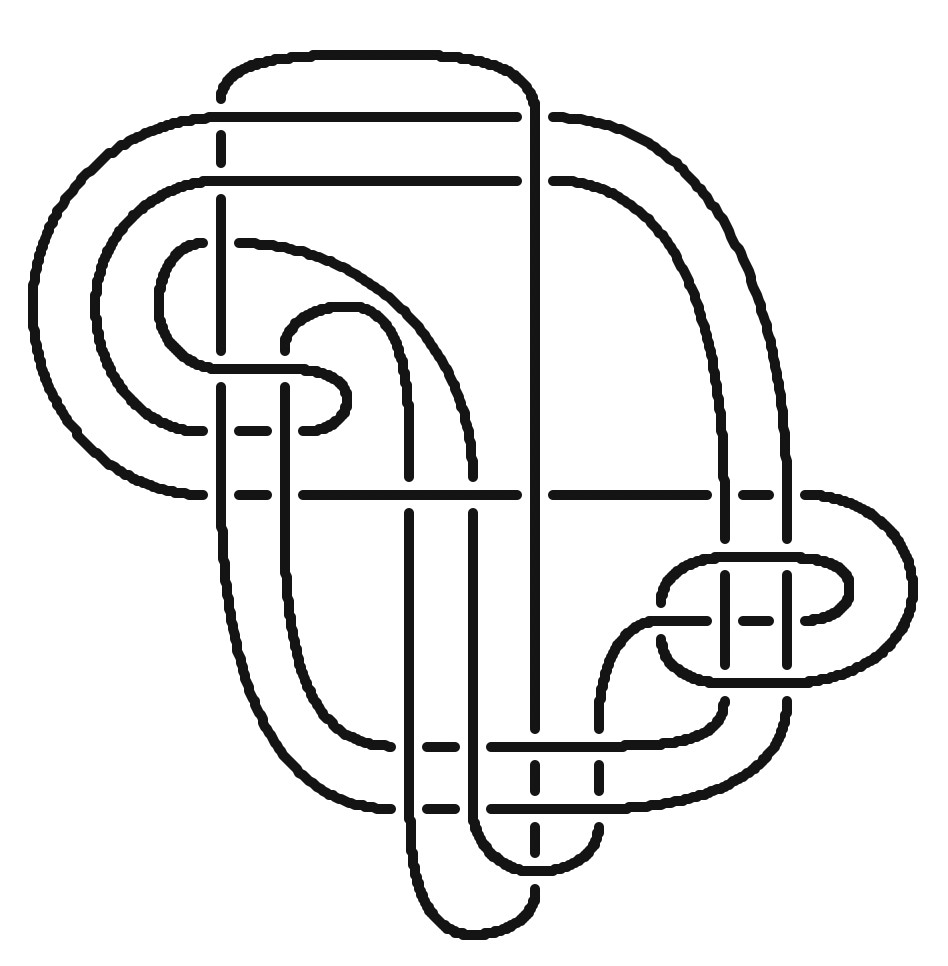}} \\
\subfloat[$K_4$]{\includegraphics[scale=.08]{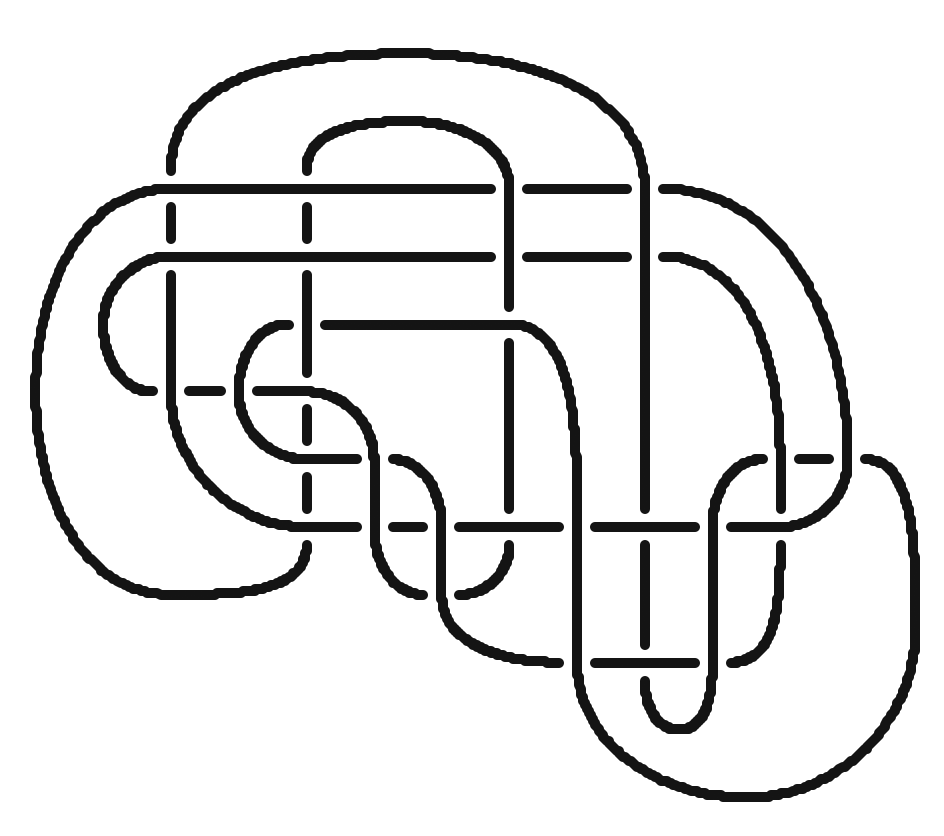}} \ \ \
\subfloat[$K_5$]{\includegraphics[scale=.08]{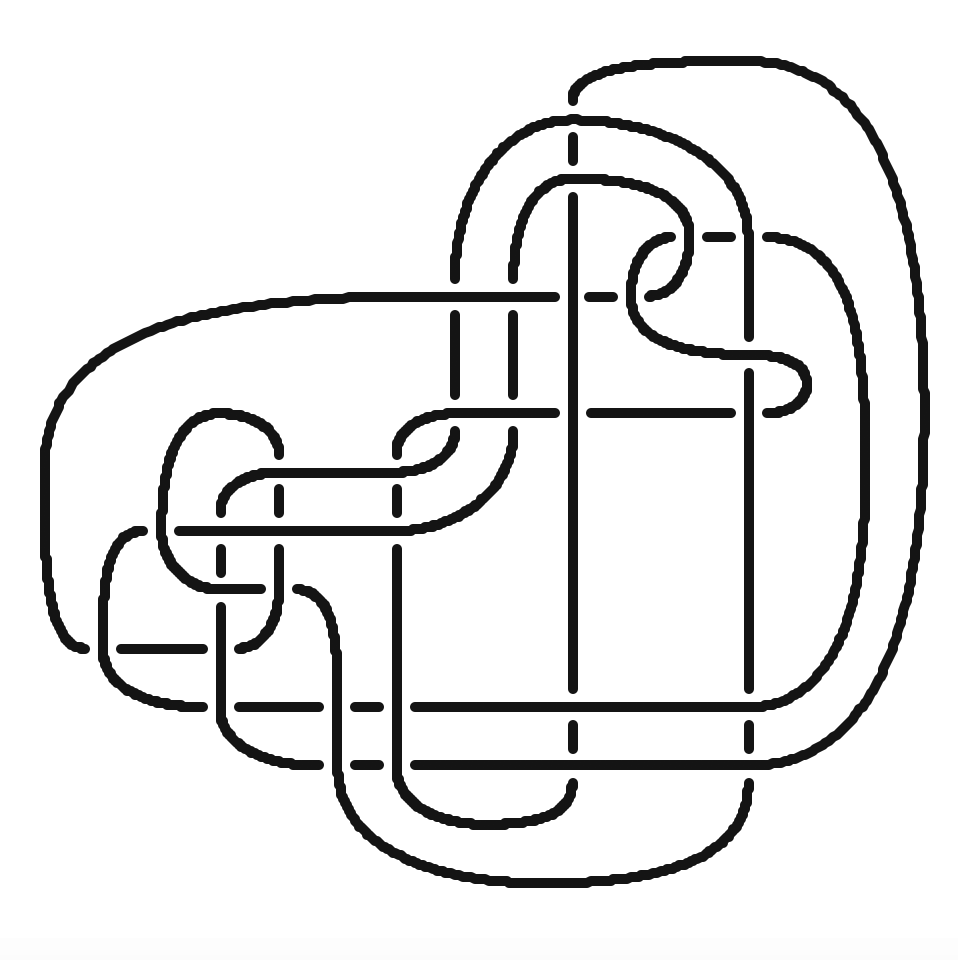}}
\caption{If any are slice, then an exotic $S^4$ exists. Figure $1$ of \cite{zero_surg_exotic}. }
\label{fig:5knots}
\end{figure}
\begin{refthm}[1.3 of \cite{zero_surg_exotic}]
If any of the knots $K_1, \dots, K_5$ of Figure \ref{fig:5knots} are slice, then an exotic $S^4$ exists.
\end{refthm}
SPC4 is a difficult, long open problem and disproving it by constructing an exotic $S^4$ is an ambitious task.
Another difficult, long open problem that might be more approachable is to construct an exotic positive definite $4$-manifold.
Historically, we have had more and earlier success constructing exotic $4$-manifolds closer to positive definite with larger topology.
In particular, it is easier to construct exotic $4$-manifolds with larger $b_2 (X) = \rnk (H_2(X))$.
We can adapt the above strategy to $\nCP$ using an adjunction inequality for the $s$-invariant in $\# n \CP$ \cite{gener_s_inv}.
We now want $K$ to be $H$-slice in $\nCP$, that is $K$ should bound a null-homologous disk $D$ in $\nCP - \inter(B^4)$.
To obstruct $H$-sliceness of $K'$ in $\nCP$, we need $s(K')<0$.
Manolescu and Piccirillo again found pairs of knots in their search that satisfied all but one of the necessary conditions.
\begin{refthm}[1.4 of \cite{zero_surg_exotic}]
If any of the knots $K_1,\dots,K_{23}$ are $H$-slice in some $\# n \CP$, then an exotic $\# n \CP$ exists\footnote{See Figure $23$ of \cite{zero_surg_exotic} for $K_6, \dots, K_{23}$}.
\end{refthm}
\subsection{Results}
The knots $K_1, \dots, K_{5}$ are good candidates to be slice.
They have Alexander polynomial $1$ and are topologically slice by Freedman \cite{Freedman}, that is they bound topologically, locally flat disks in $B^4$.
Therefore, all obstructions to topological sliceness automatically vanish on these knots.
Many smooth concordance invariants such as Ozsv\'{a}th-Szab\'{o}'s $\tau$-invariant and Rasmussen's $s$-invariant also vanish on them.
In addition, the knots $K_1', \dots, K_5'$ would be slice in a homotopy $B^4$ and many of the necessary invariants vanish on these knots.
At first thought, one would need new slice obstructions that are stronger than those currently available.
Despite all of this we are able to show the following:
\begin{thm}
\label{thm:notslice}
The knots $K_1, \dots, K_{23}$ are not slice.
Furthermore, these knots are not $H$-slice in any $\# n \CP$.
\end{thm}
The difficulty here is the vanishing of invariants that obstruct sliceness or $H$-sliceness of $K_i$.
This is the same problem that arises when trying to determine sliceness of the Conway knot.
For the Conway knot, Piccirillo finds a knot $K$ that shares a zero trace with the Conway knot.
Since sliceness is determined by the zero trace, the Conway knot is slice if and only if $K$ is slice.
Calculating $s(K)$ shows that $K$ is not slice and therefore the Conway knot is not slice \cite{Conway_knot}.
One might hope to extend the zero surgery homeomorphism $S^3_0(K_i) \rightarrow S^3_0(K_i')$ to a zero trace diffeomorphism.
For many of these pairs, the zero surgery homeomorphisms do not extend and it appears they may never have homeomorphic traces.
Without a trace diffeomorphism, we can't access the trace embedding lemma to identify sliceness of $K$ and $K'$.
Instead, we extend $S^3_0(K_i) \rightarrow S^3_0(K_i')$ to a diffeomorphism of the traces after blowing up.
This allows us to relate their slice properties \textit{stably} and work with $H$-sliceness of $K_i'$ instead of the difficult $K_i$.

Manolescu and Piccirillo considered an infinite six parameter family of zero surgery homeomorphisms.
They found the knots $K_1,\dots,K_{23}$ by checking $3375$ zero surgery homeomorphisms in this family.
One might try to expand the search and consider more pairs from this family.
We show that such an effort would be in vain and prove a stronger version of Theorem \ref{thm:notslice}.
\begin{thm}
\label{thm:MP_family_fullversion}
Let $(K,K')$ be a pair of knots with homeomorphic zero surgeries from the Manolescu-Piccirillo family.
\begin{enumerate}
  \item If $K$ is $H$-slice in some $\nCP$, then $s(K') \ge 0$.
  \item If $K$ is $H$-slice in some $\nCPbar$, then $s(K') \le 0$.
  \item If $K$ is slice, then $s(K')=0$.
\end{enumerate}
\end{thm}
This is proved in a more general form in Theorem \ref{thm:small_RBG}.
This theorem rules out finding an exotic $S^4$ (or $\nCP$) using the $s$-invariant and zero surgery homeomorphisms from the Manolescu-Piccirillo family.
This does not show the stronger statement that such $(K,K')$ can not be used to construct an exotic $S^4$.
In principal, a $(K,K')$ from the Manolescu-Piccirillo family could still have $K$ slice and $K'$ not slice.
Such a pair would exhibit an exotic $S^4$, but this theorem shows that the $s$-invariant $s(K')$ would not obstruct sliceness and detect exoticness.

In Theorem \ref{thm:full_adj_version}, we attempt to generalize the above theorem assuming a conjectural inequality for Rasmussen's $s$-invariant.
We establish conditions on when these methods apply and would rule out using the $s$-invariant with zero surgery homeomorphisms to construct an exotic $S^4$ or $\nCP$.
These methods are not special to the $s$-invariant and might also apply to other obstructions to $H$-sliceness in $\nCP$ or other $4$-manifolds.
As new and stronger concordance invariants are inevitably constructed, there will surely be new attempts to construct exotic $4$-manifolds using $H$-sliceness and zero surgery homeomorphisms.
Such hypothetical future attempts will likely need to revisit this work.

Our methods do not apply to all zero surgery homeomorphisms and leaves hope that the $s$-invariant could be used to find an exotic $S^4$ or $\nCP$.
We construct an infinite family of zero surgery homeomorphisms which are not susceptible to our methods.
These zero surgery homeomorphisms are not a serious attempt at constructing an exotic $S^4$ or $\nCP$.
Instead they are an illustration of the continued viability of Manolescu and Piccirillo's approach and an invitation to the topological community to continue it.

Manolescu and Piccirillo also consider pairs of knots $(K,K')$ which have have homeomorphic $0$-surgeries, $K$ is slice, and $K'$ has indeterminate sliceness.
They give a family of knots $\{J_n \ | \ n \in \Z \}$ related by annulus twist homeomorphisms $\phi_n: S^3_0(J_0) \rightarrow S^3_0(J_n)$.
$J_0$ bounds a ribbon disk and this gives a family of homotopy spheres $Z_n = E(D) \cup_{\phi_n} -X_0(J_n)$.
\begin{thm}
The homotopy $4$-spheres $Z_n$ are all diffeomorphic to $S^4$ (and therefore each $J_n$ is slice).
\end{thm}
To prove this, we draw these manifolds upside down as $X_0(J_n) \cup -E(D)$.
Drawing the exterior upside down directly with the standard algorithm is difficult and results in a complicated diagram.
Instead we will describe an algorithm for any ribbon knot $K$ bounding a ribbon disk $D$, how to draw a Kirby diagram of $S^4$ as $X_0(K) \cup -E(D)$.
We can then use this to draw a Kirby diagram of $-Z_n$ showing the trace embedding $X_0(J_n) \subset-Z_n$.
Using this diagram of $-Z_n$, we then show that each $Z_n$ is standard.
\subsection{Conventions}
All manifolds are smooth and oriented.
Any embeddings or homeomorphism are orientation preserving.
Boundaries are oriented with outward normal first.
All homology groups have integral coefficients.
\subsection{Acknowledgments}
The author would like to thank Ciprian Manolescu and Lisa Piccirillo for helpful correspondences as well as for allowing the author to use the images in Figure \ref{fig:5knots}.
The author would also like to thank his advisors Bob Gompf and John Luecke for their help and support.
As noted in \cite{zero_surg_exotic}, some cases of the above results were already established by others.
Dunfield and Gong showed that $K_6, \dots, K_{21}$ are not slice using their program to compute twisted Alexander polynomial \cite{twist_alexpoly_software} and Kyle Hayden showed that $Z_1$ is standard

\section{Preliminaries}
\subsection{H-slice Knots and Zero Surgery Homeomorphisms}
\label{sec:MP_background}
We will need to recall Manolescu and Piccirillo's proposed construction of an exotic $S^4$ or $\nCP$.
To simplify the discussion, we combine these cases into one and define $\# 0 \CP$ to be $S^4$ via the empty connected sum.
Whenever $\nCP$ appears it will be implicit that $n \ge 0$ and likewise with $\nCPbar$.

Let $X$ be a smooth, closed, oriented $4$-manifold and let $X^\circ = X-\inter(B^4)$.
\begin{mydef}
A knot $K \subset S^3$ is said to be $H$-slice in $X^{\circ}$ or $X$ if $K$ is the boundary of a smoothly, properly embedded disk $D$ in $X^\circ$ such that $[D] = 0 \in H_2(X^\circ, \partial X^\circ)$.
\end{mydef}
$H$-sliceness is a generalization of sliceness: a knot is slice (in $B^4$) if and only if it is $H$-slice in $S^4$.
Recall that the $k$-trace $X_k(K)$ of $K$ is obtained by attaching a $2$-handle to $B^4$ along $K$ with framing $k$.
The classical trace embedding lemma asserts a knot $K$ is slice if and only if the zero trace $X_0(K)$ smoothly embeds in $S^4$.
We have an analogous statement for a knot to be $H$-slice in $X$.
\begin{lem}[H-slice Trace Embedding Lemma, Lemma 3.5 of \cite{zero_surg_exotic}]
A knot $K$ is $H$-slice in $X$ if and only if $-X_0(K)$ smoothly embeds in $X$ by an embedding that induces the zero map on second homology.
\label{lem:Hslice_TEL}
\end{lem}
Suppose $K$ is $H$-slice in $X$ with $H$-slice disk $D \subset X^\circ$ and there is a zero surgery homeomorphism $\phi : S^3_0(K) \rightarrow S^3_0(K')$.
Let $\nu(D)$ be a tubular neighborhood of $D$ and let the exterior of $D$ be $E(D) = X^\circ - \nu(D)$.
The exterior naturally has boundary $S^3_0(K)$ and we can define $X' = E(D) \cup_\phi -X_0(K')$.
$K'$ is $H$-slice in $X'$ by Lemma \ref{lem:Hslice_TEL} and if $X$ is simply connected, $X'$ is homeomorphic to $X$ (Lemma $3.3$ of \cite{zero_surg_exotic}).
If $K'$ is not $H$-slice in $X$, then $H$-sliceness of $K'$ smoothly distinguishes $X'$ from $X$.

Constructing an exotic $\nCP$ with zero surgery homeomorphisms sounds promising, but there are difficulties with this approach which have only recently been resolved.
The first challenge was overcoming the Akbulut-Kirby conjecture which asserts that knots with homeomorphic zero surgeries are concordant.
$H$-sliceness is preserved by concordance and therefore this construction would be more difficult than producing counterexamples to the Akbulut-Kirby conjecture.
Fortunately, Yasui disproved the Akbulut-Kirby conjecture in $2015$.
In doing so, he showed that concordance invariants, such as the Ozsv\'{a}th-Szab\'{o} $\tau$-invariant or Rasmussen's $s$-invariant, could distinguish knots in concordance that share a zero surgery \cite{yasui_akb_kirby_conj}.

This brings us to the second difficulty with this strategy.
We need to obstruct $H$-sliceness of $K'$ in the standard $\nCP$ without obstructing $H$-sliceness of $K$ in a homotopy $\nCP$.
Obstruction from gauge and Floer theoretic concordance invariants, like the $\tau$-invariant, tend to apply in any homotopy $\nCP$ \cite{tau_adj_ineq}.
In particular, such invariants always vanish on knots slice in a homotopy $S^4$.
However, Rasmussen's $s$-invariant does provide an obstruction to $H$-sliceness in $\nCP$ that may not hold in a homotopy $\nCP$.
\begin{lem}[Theorem $1.8$ of \cite{gener_s_inv}]
\label{lem:HSlice_s_inv}
If $\Sigma \subset \nCP - (\inter(B^4) \sqcup \inter(B^4))$ is a null homologous, oriented cobordism from a link $L_1$ to $L_2$ with each component of $\Sigma$ having a boundary component in $L_1$, then $s(L_2) -s(L_1) \ge \chi(\Sigma)$.
In particular, if a knot $K$ is $H$-slice in some $\nCP$, then $s(K) \ge 0$.
\end{lem}
\noindent By reversing orientation, we see that if $K$ is $H$-slice in $\nCPbar$, then $s(K) \le 0$.
Furthermore, if $K$ is $H$-slice in $\nCPbar$ and $\nCP$ for some $n$, then $s(K) = 0$.
Such knots are called \textit{biprojectively H-slice} by Manolescu and Piccirillo.
These include all slice knots and some non-slice knots like the figure eight knot.

Putting this together, we can construct an exotic $\nCP$ if we have a pair of knots $(K,K')$ such that
\[K \text{ is H-slice in }\nCP,   \ \ \ s(K')<0, \ \ \ \phi: S^3_0(K) \rightarrow S^3_0(K')\]
\noindent For $H$-sliceness in $\# 0\CP = S^4$, i.e. standard sliceness, we could also consider $s(K') \neq 0$ to obstruct sliceness.
However, Manolescu and Piccirillo focus on negative $s(K')$ and in their search find no viable examples with positive $s(K')$.

Recall that a framed knot is a knot $K$ in $S^3$ together with a framing $k \in \Z$.
We will denote a framed knot by $(K,k)$ and extend this naturally to framed links.
To conduct their search, Manolescu and Piccirillo need a source of zero surgery homeomorphisms and so they define special RBG-links.
\begin{mydef}
A special RBG-link $L = (R, r) \cup (B,0) \cup (G,0) \subset S^3$ is a three component integrally framed link where $R$ has framing $r \in \Z$, $B$ and $G$ have framing $b=g=0$.
Furthermore, surgery on this framed link has $H_1(S^3_{r,0,0}(R,B,G)) = \Z$ and if $\mu_R$ is a meridian of $R$, there exist link isotopies
\[R \cup G \cong R \cup \mu_R \cong R \cup B\]
\end{mydef}
Given a special RBG-link we can define a pair of knots and a zero surgery homeomorphism between them.
The following proposition and its proof is the first half of Theorem $1.2$ of \cite{zero_surg_exotic} for a special RBG-link.
We reproduce it here because understanding the special RBG-link construction will be \textbf{fundamental} to proving our key lemmas.
\begin{figure}
\centering

\subfloat[]{{
   \fontsize{10pt}{12pt}\selectfont
   \def\svgwidth{2in}
   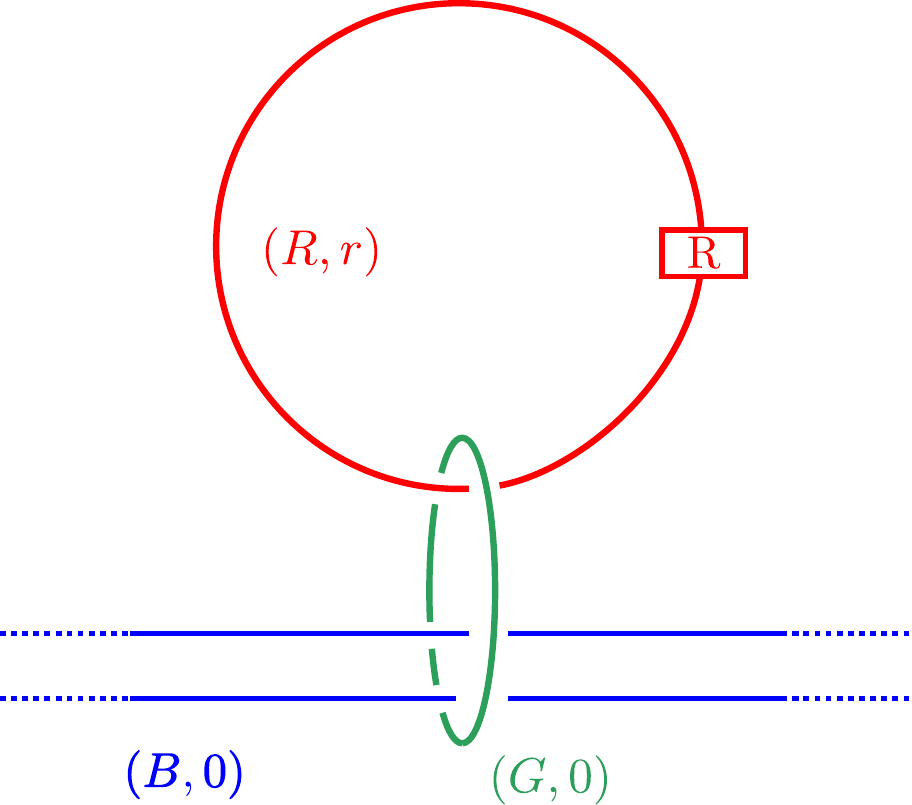
}
\label{fig:slamdunk1}
} \ \
\subfloat[]{{
   \fontsize{10pt}{12pt}\selectfont
   \def\svgwidth{2in}
   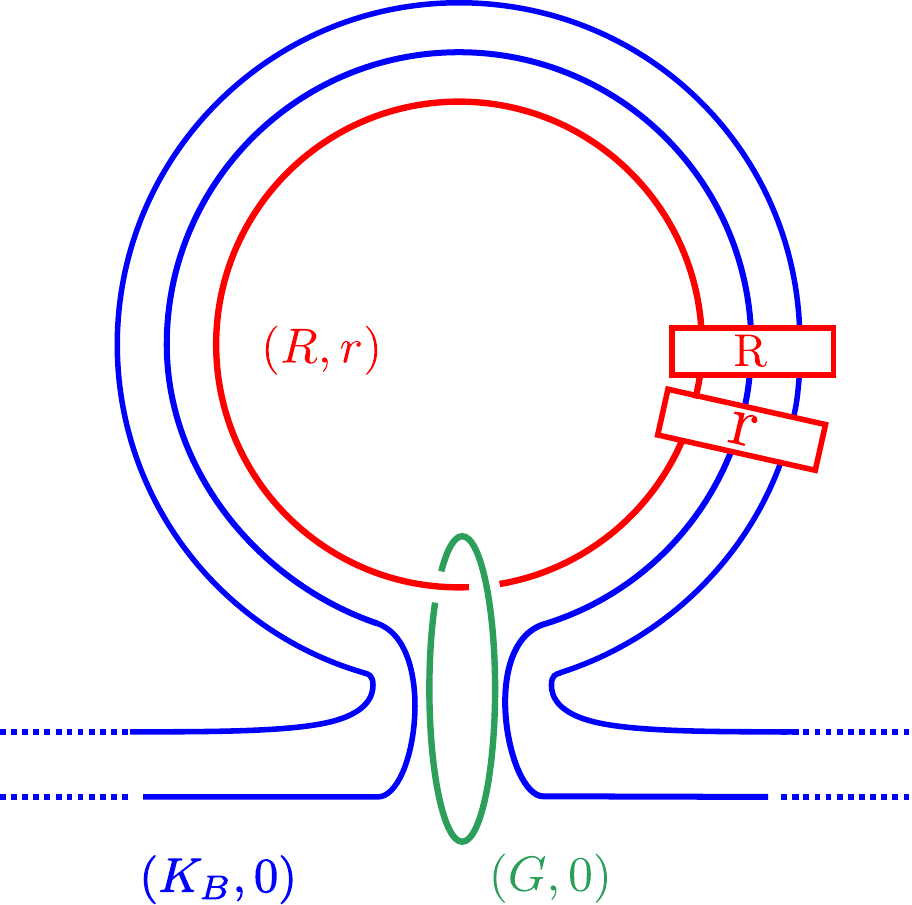
}
\label{fig:slamdunk2}
}
\caption{Initial slide used to exhibit $K_B$}
\label{fig:slamdunk_all}

\end{figure}
\begin{prop}
\label{prop:special_RBG_homeo}
For any special RBG-link $L$, there is an associated pair of knots $K_B$ and $K_G$ and a homeomorphism $\phi_L : S^3_0(K_B) \rightarrow S^3_0(K_G)$.
\end{prop}
\begin{proof}
The assumption that $(G,0)$ is zero framed meridian of $(R,r)$ implies there is a slam dunk homeomorphism $\psi_B: S^3_{r,0}(R,G) \rightarrow S^3$.
Let $K_B = \psi_B(B)$, the slam dunk homeomorphism takes a framing on $B$ to a framing on $K_B$ that surgers to $S^3_{r,b,g}(R,B,G)$.
The assumption on homology implies that this must be the zero framing.
Reusing notation, the slam dunk on $(R,r) \cup (G,0)$ induces a homeomorphism $\psi_B : S^3_{r,b,g}(R,B,G) \rightarrow S^3_0(K_B)$.
We can do the same with $(G,0)$ by slam dunking $(R,r) \cup (B,0)$ to get a homeomorphism $\psi_G : S^3_{r,b,g}(R,B,G) \rightarrow S^3_0(K_G)$ and the desired homeomorphism is then $\phi_L = \psi_B \circ \psi_G^{-1} : S^3_0(K_B) \rightarrow S^3_0(K_G)$.
\end{proof}
Given a diagram of $L$, we can perform this construction diagrammatically.
We take $L$ to be a surgery diagram of $S^3_{r,b,g}(R,B,G)$ and slam dunk $(R,r) \cup (G,0)$.
To do this, first isotope $(G,0)$ into meridianal position so $(G,0)$ bounds a disk $\Delta_G$.
This disk intersects $(B,0)$ in some number of points as shown in Figure \ref{fig:slamdunk1}.
Slide $(B,0)$ over $(R,r)$ so that it no longer intersects $\Delta_G$ as shown in Figure \ref{fig:slamdunk2}.
To finish the slam dunk, delete $(R,r)$ and $(G,0)$ leaving $(K_B,0)$.
\subsection{Projective Slice Framings}
\label{sec:proj_slice_framing}
Let $W$ be a smooth, closed, oriented $4$-manifold.
If $D$ is a disk properly embedded in $W^\circ$, then $D$ has a well defined tubular neighborhood $\nu(D) = D \times \mathbb{R}^2$.
Then $D = D \times \{0\}$ has a parallel pushoff $D^* = D \times \{p\} \subset \nu(D)$ for some nonzero $p$.
$K^* = \partial D^*$ is a knot parallel to $K$ and defines a framing on $K$.
\begin{mydef}
A framed knot $(K, k)$ in $S^3$ is said to be slice in $W$ or $W^\circ$ if $K$ is the boundary of a disk $D$ smoothly, properly embedded in $W$ which induces the framing $k$ on $K$.
\end{mydef}
If we say $K$ is $k$-slice in $W$ with $k \in \Z$, then we mean $(K,k)$ is slice in $W$.
This framing $k$ will be equal to the negative of the self intersection number of $D$, i.e. $k = -[D] \cdot [D]$.
The exterior $E(D) = W^\circ - \nu(D)$ of $D$ has boundary $\partial E(D)$ naturally identified with $S^3_k(K)$.
We can view the deleted $\nu(D)$ and $int(B^4)$ as a trace and get a trace embedding lemma.
\begin{lem}[Framed Trace Embedding Lemma, Lemma 3.3 of \cite{tracelemma_spine}]
\label{lem:framed_TEL}
A framed knot $(K,k)$ in $S^3$ is smoothly slice in $W$ if and only if $-X_k(K)$ smoothly embeds in $W$.
\end{lem}
This will allow us later to construct framed slice disks by finding trace embeddings.
We will be working with framed slice disks in $\nCPbar$ and in this setting, it is often more practical to construct the disks directly.
To construct knots $H$-slice in some $\nCPbar$, one can use a full positive twist along algebraically zero strands as in Lemma $3.2$ of \cite{zero_surg_exotic}.
This generalizes to framed sliceness, but now we need to keep track of how the framing changes.
\begin{lem}
\label{lem:frame_CP_cob}
Suppose $(K,k)$ is the framed boundary of a disk $D \subset W^\circ$ and $\Delta$ a disk embedded in $S^3$ intersecting $K$ transversely in $\ell$ points counted with sign.
Let $K_+$ be a knot obtained from $K$ by performing a positive full twist through $\Delta$, then $(K_+,k+\ell^2)$ is slice in $W \# \CPbar$.
\end{lem}
\begin{proof}
Attach a $-1$ framed $2$-handle to $W^\circ$ along $\partial \Delta$ to get $W^\circ \cup_{(\partial \Delta, -1)} 2h = (W \# \CPbar)^\circ$.
Then $D$ in $W^\circ \cup_{(\partial \Delta, -1)} 2h$ has boundary $(K,k) \subset S^3_{-1}(\partial \Delta)$ and a blowdown of $(\partial \Delta,-1)$ identifies it with $(K_+,k+\ell^2)$.
\end{proof}
The above lemma can be used to show that an arbitrary knot $K$ will be slice in some $\nCPbar$.
First observe that a crossing change can be realized as a positive twist on two strands.
Then a sequence of crossing changes that unknot $K$ can be used to construct a slice disk $D$ for $K$ with some framing in $\nCPbar$.
Our arguments require checking that certain framed knots associated to a zero surgery homeomorphism are slice in some $\nCPbar$ or $\nCP$.
To quantify this we define the projective slice framings.
\begin{mydef}
Let $K$ be a knot in $S^3$.
The positive projective slice framing $\PF_+(K)$ of $K$ is the smallest framing $k$ such that $(K,k)$ is slice in some $\nCPbar$.
The negative projective slice framing $\mathbb{PF}_-(K)$ is defined analogously.
\end{mydef}
Homological considerations imply that a strictly negative framed knot cannot be slice in a negative definite $4$-manifold and therefore $\PF_+(K) \ge 0$.
Furthermore, in a negative definite $4$-manifold, the only homology class with self intersection number zero is the zero homology class.
\begin{lem}
     $\PF_+(K) = 0$ if and only if $K$ is $H$-slice in some $\nCPbar$.
     $\PF_-(K) = 0$ if and only if $K$ is $H$-slice in some $\nCP$.
     $\PF_+(K) = \PF_-(K) = 0$ if and only if $K$ is biprojectively $H$-slice.
\label{lem:proj_Hslice}
\qed
\end{lem}
Once we have one of these framings, Lemma \ref{lem:frame_CP_cob} immediately realizes any framing larger than $\PF_+(K)$.
Take $D$ a disk realizing $\PF_+(K)$ and $\Delta$ a meridianal disk of $K$.
\begin{cor}
If $k \ge \PF_+(K)$, then $(K,k)$ is slice in some $\nCPbar$.
If $k \le \PF_-(K)$, then $(K,k)$ is slice in some $\nCP$.
\qed
\label{cor:atleast_PF}
\end{cor}
We can construct framed slice disks in $\nCPbar$, but we would also like to have lower bounds on these framings as well.
To do so, we can use the adjunction inequality for the Ozsv\'{a}th-Szab\'{o} $\tau$-invariant.
\begin{lem}[Theorem 1.1 of \cite{tau_adj_ineq}]
\label{lem:tau_adj}
Let $W$ be a smooth $4$-manifold with negative definite intersection form, $b_1(W) = 0$, and $\partial W = S^3$.
Let $e_1,\dots,e_n$ be an orthonormal basis for $H_2(W)$ and if $\alpha = s_1 e_1 + \dots s_n e_n$, then let $|\alpha| = |s_1|+\dots +|s_n|$.
If $\Sigma$ is smooth, properly embedded surface in $W$ with boundary $\partial W = K$, then we have the following inequality:
\[2 \tau(K) + |[\Sigma]| + [\Sigma] \cdot [\Sigma] \le 2g(\Sigma)\]
\end{lem}
\begin{cor}
\label{cor:PF_bounds_tau}
\[\PF_-(K) + \sqrt{|\PF_-(K)|} \le 2\tau(K) \le  \PF_+(K) - \sqrt{\PF_+(K)}\]
\end{cor}
\begin{proof}
Let $\Sigma$ be a disk in $\nCPbar -int(B^4)$ with framed boundary $(K,k)$ and let $x = 2\tau(K)$.
We have $[\Sigma] \cdot [\Sigma] = -k$ and so $x + |[\Sigma]| \le k$ by Lemma \ref{lem:tau_adj}.
We see that $k = s_1^2+\dots+s_n^2 \le (|s_1| + \cdots +|s_n|)^2 = |[\Sigma]|^2$, therefore $\sqrt{k} \le |[\Sigma]|$ and $x \le k - \sqrt{k}$.
\end{proof}
There is an analogous adjunction inequality conjectured for the $s$-invariant.
\begin{conj}[Conjecture 9.8 of \cite{gener_s_inv}]
\label{conj:s(k)_adj}
If $\Sigma$ is smooth, properly embedded surface in $W = \nCPbar -int(B^4)$ with boundary $\partial \Sigma = K$, then we have the following inequality:
\[s(K) + |[\Sigma]| + [\Sigma] \cdot [\Sigma] \le 2g(\Sigma)\]
\end{conj}
\cite{gener_s_inv} proved this in the nullhomologous case and conjectured this in analogy to the adjunction inequality for $\tau(K)$.
We replace $2\tau(K)$ with $s(K)$ like their slice genus bounds, but we limit this conjecture to $W = \nCPbar - \inter(B^4)$.
It's not clear how to approach Conjecture \ref{conj:s(k)_adj} for arbitrary negative definite $W$ or even if it should hold in such $W$.
However, if this conjecture holds, we can replace $x = 2 \tau(K)$ with $s(K)$ throughout the proof of Corollary \ref{cor:PF_bounds_tau}.
When $K$ is $(-1)$-slice in some $\nCP$, we conjecturally get the same restriction on $s(K)$ as we did when $K$ is $H$-slice is some $\nCP$.
\begin{conj}
If $K$ is $(-1)$-slice in some $\nCP$, then $s(K) \ge 0$.
\label{conj:PF_is_1_bound}
\end{conj}
\section{Construction of Trace Embeddings}
\subsection{Obstructing H-sliceness of \texorpdfstring{$K_1,\dots,K_{23}$}{} in \texorpdfstring{$\nCP$}{}}
\label{sec:obHslice}
In this subsection, we show that $K_1,\dots,K_{23}$ are not slice, nor are they H-slice in any $\nCP$.
To do this, we will relate H-sliceness of $K_i$ in $\nCP$ stably with $K_i'$ after a connected sum with some $\rCP$.
We see this first at the level of traces as a stable trace diffeomorphism.
\begin{lem}
\label{lem:stable_trace_diffeo}
Let $W$ be a smooth, closed, oriented $4$-manifold and $L = (R,r) \cup (B,0) \cup (G,0)$ be a special RBG-link with associated zero surgery homeomorphism $\phi_L : S^3_0(K_B) \rightarrow S^3_0(K_G)$.
If $(R,r)$ is slice in $W$, then $\phi_L$ extends to a diffeomorphism $\Phi_L: X_0(K_B) \# W \cong X_0(K_G) \# W$
\end{lem}
This is a generalization of techniques \cite{akb_2dim,shake_slice_genus,Conway_knot} used to construct trace diffeomorphisms.
The $(R,r) = (U,0)$ slice in $B^4$ case of the lemma is the dualizable link construction and its proof is insightful here.
Construct a Kirby diagram by placing a dot on $R$ and attaching $2$-handles to $B$ and $G$.
Cancelling the dotted $R$ with $B$ or $G$ diagrammatically is a slam dunk resulting in $X_0(K_G)$ and $X_0(K_B)$ respectively.
The framed sliceness of $(R,r)$ in $W$ allows us to ``dot'' $(R,r)$ and proceed in a similar manner.
\begin{proof}
Let $Z$ be the $4$-manifold obtained from $W^\circ$ by removing a neighborhood $\nu(D)$ of a slice disk $D$ for $(R,r)$ and attaching $2$-handles to $(B,0)$ and $(G,0)$.
This description naturally identifies $\partial Z$ as $S^3_{r,0,0}(R,B,G)$.
The $2$-handle attached to $(G,0)$ fills $E(D)$ because $(G,0)$ is isotopic to $(\mu_R,0)$.
This is a diffeomorphism $\Psi_B : E(D) \cup_{(G,0)} 2h \rightarrow W^\circ$ extending the slam dunk homeomorphism $\psi_B: S^3_{r,0}(R,G) \rightarrow S^3$.
The $2$-handle that was attached to $(B,0)$ is now attached to $\psi_B(B,0) = (K_B,0)$ and therefore $\Psi_B$ induces a diffeomorphism of $Z$ and $W^\circ \cup_{(K_B,0)} 2h$ .
Observe that $W^\circ = B^4 \# W$ and $W^\circ \cup_{(K_B,0)} 2h$ is simply $X_0(K_B) \# W$.
Reusing notation, we have a diffeomorphism from $\Psi_B: Z \rightarrow X_0(K_B) \# W$ extending the homeomorphism $\psi_B: S^3_{r,0,0}(R,B,G) \rightarrow S^3_0(K_B)$.
Repeating this with $G$ instead and composing, we have the desired diffeomorphism $\Phi_L = \Psi_B \circ \Psi_G^{-1} : X_0(K_B) \# W \cong X_0(K_G) \# W$ extending the homeomorphism $\phi_L = \psi_B \circ \psi_G^{-1}$.
\end{proof}
When a zero surgery homeomorphism extends to a zero trace diffeomorphism, the $H$-slice trace embedding lemma identifies the $H$-sliceness of the two knots.
If the zero surgery homeomorphism doesn't extend, Lemma \ref{lem:stable_trace_diffeo} sometimes allows us to extend to a diffeomorphism after a connected sum with $-W$.
This makes the trace embedding lemma available and allows us to relate $H$-sliceness of the two knots.
\begin{cor}
\label{cor:HtpyX_stable_diff}
Let $X$ be a smooth, closed, oriented $4$-manifold.
Suppose $K_B$ is $H$-slice in $X$ and $K_G$ is then $H$-slice in $X' = E(D) \cup_{\phi_L} -X_0(K_G)$.
If $(R,r)$ is slice in $W$, then $X' \# -W$ is diffeomorphic to $X \# -W$ and therefore $K_G$ is $H$-slice in $X \# -W$.
\qed
\end{cor}
Now we are ready to prove our first theorem.
\begin{thm}
\label{main_thm}
The knots $K_1, \dots, K_{23}$ are not slice nor are they $H$-slice in any $\nCP$.
\end{thm}
\begin{proof}
By looking at Figure $13$ and Table $1$ of \cite{zero_surg_exotic}, we see that each $K_i$ arise from a special RBG-link where $R = U$ and $r = a + b \ge 0$.
We have $(R,r) = (U,r)$ is slice in $W = \rCPbar$ and if $K_i$ is $H$-slice in $\nCP$, then $K_i'$ is $H$-slice in $\#(n+r) \CP$ by Corollary \ref{cor:HtpyX_stable_diff}.
These knots have $s(K_i') = -2$ which contradicts Lemma \ref{lem:HSlice_s_inv} and therefore $K_i$ could not have been $H$-slice in $\nCP$ in the first place.
\end{proof}
\begin{remark}
\label{rem:blowup}
For the knots in question, the diffeomorphism $\Phi_L: X_0(K_G) \rCPbar \cong X_0(K_B) \rCPbar$ can also be seen diagrammatically.
Take $L$ and perform negative blow ups to $(R,r)$ turning it into a zero framed unknot $(R,0)$ with $r$ meridians $(R_1,-1), \dots, (R_r,-1)$.
This new framed link still describes the same zero surgery homeomorphism.
Slam dunk $(R,0)$ with $(B,0)$ or $(G,0)$ and blow down $(R_1,-1), \dots, (R_r,-1)$ to get $(K_G,0)$ or $(K_B,0)$ respectively (think of this as an RBG-link generalized to have $R$ be a framed link).
Form a Kirby diagram by putting a dot on $R$ and attaching $2$-handles to the other components.
Now cancelling the dotted $(R,0)$ with $(B,0)$ or $(G,0)$ and sliding $(R_1,-1), \dots, (R_r,-1)$ away results in $X_0(K_G) \rCPbar$ or $X_0(K_B) \rCPbar$ respectively.
\end{remark}
We have dispensed of the main question relatively quickly and have shown that $K_1, \dots, K_{23}$ are not slice or $H$-slice in any $\nCP$.
However, to prove this we needed that the associated special RBG-links all had $r \ge 0$.
If we had negative $r < 0$ instead, then $(R,r)$ would be slice in $\# |r| \CP$.
If $K$ was $H$-slice in $\nCP$, then $K'$ would be $H$-slice in $\nCP \# |r| \CPbar$ by Corollary \ref{cor:HtpyX_stable_diff}.
This would not contradict $s(K')<0$ and our proof would not work.
It seems quite mysterious that we had this necessary condition on $r$ for all $23$ pairs of knots.
Fortunately, we can explain this and do so in the following subsection.
Before we proceed, we take a short detour to provide a generalization of Corollary \ref{cor:HtpyX_stable_diff} from special RBG-links to arbitrary zero surgery homeomorphisms.
One can safely skip to the next subsection, but this method may offer some useful benefits in practice.
\begin{lem}
\label{lem:arb_zero_surg_hslice}
Let $\zsg$ be a zero surgery homeomorphism and represent $\phi^{-1}(\mu_{K'},0) \subset S^3_0(K)$ as some framed knot $(m,k)$ in $S^3$. 
If $K$ is $H$-slice in $X$ and $(m,k)$ is slice in some $4$-manifold $W$, then $K'$ is $H$-slice in $X \# -W$.
\end{lem}
\begin{proof}
By turning the handle decomposition of $X_0(K')$ upside down, one can obtain $X'$ from $E(D)$ by attaching a $2$-handle to $(m,k)$ and capping off with a $4$-handle.
We can regard $(X')^\circ$ as $X'$ with a $4$-handle deleted and $(X')^\circ = E(D)\cup_{(m,k)} 2h$.
This can be used to build the following sequence of inclusions:
\[-X_0(K') \subset (X')^\circ = E(D)\cup_{(m,k)} 2h \subset X^\circ \cup_{(m,k)} 2h = X \# X_k(m) \subset X \# -W\]
The first inclusion comes from the $H$-slice trace embedding lemma and the second inclusion is induced by $E(D) \subset X^\circ$.
As we observed in the proof of Lemma \ref{lem:stable_trace_diffeo}, $X^\circ \cup_{(m,k)} 2h$ is simply $X \# X_k(m)$.
The final inclusion comes from the framed trace embedding lemma with $(m,k)$ slice in $W$.
The composition is a nullhomologous trace embedding of $-X_0(K')$ into $X \# -W$ and therefore $K'$ is $H$-slice in $X \# -W$.
\end{proof}
For a special RBG-link homeomorphism, one can take $(m,k)$ to be $(R,r)$.
This recovers the conclusion of Corollary \ref{cor:HtpyX_stable_diff} and gives another proof of Theorem \ref{main_thm}.
However, we do not get a stable diffeomorphism $X' \# -W \cong X \# -W$ like in Corollary \ref{cor:HtpyX_stable_diff}.
The advantage of this is that it does not need a special RBG-link and has the further benefit that the input data is more malleable.
It seems non-trivial to find special RBG-links representing a given zero surgery homeomorphism with a different $(R,r)$.
However, $(m,k)$ was a choice of diagrammatic representative of $\phi^{-1}(\mu_{K'},0) \subset S^3_0(K)$ and can be modified via slides over $K$.
\subsection{The Manolescu-Piccirillo Family}
\label{sec:signs}
This section is devoted to explaining why $r$ was positive for all of Manolescu and Piccirillo's knot pairs.
We observed after proving Theorem \ref{main_thm} that we needed the relevant special RBG-link to have $r \ge 0$ for our proof to work.
Having positive $r \ge 0$ for all $23$ knot pairs seems unlikely and it would be quite unsatisfactory if a necessary hypothesis in our proof was left unexplained.
We investigate this by widening our view and considering the infinite six parameter family of special RBG-links $\MPfam$ that these knot pairs arose from.
These $\MPfam$ were constructed by Manolescu and Piccirillo as a source of candidates for their exotic $\nCP$ construction.
We prove that this would always occurs for any member $\MPfam$ of the Manolescu-Piccirillo family: if $r < 0$, then $s(K),s(K') \ge 0$.
This allows us to prove a strong generalization of Theorem \ref{main_thm} from $K_1,\dots,K_{23}$ to any $(K,K')$ coming from some $\MPfam$.

To show that $s(K),s(K')$ are non-negative when $r <0$, we construct and analyze slice disks for $(K,-1)$ and $(K',-1)$ in $\# |r| \CP$.
We build these slice disks by exploiting two properties of the special RBG-links $L(a,b,c,d,e,f)$.
The first was that they had $R=U$ and the second was that they were \textit{small}.
\begin{mydef}
A small RBG-link is a special RBG-link $L$ such that
\begin{itemize}[itemsep=0em]
  \item $B$ bounds a properly embedded disk $\Delta_B$ that intersects $R$ in exactly one point, and intersects $G$ in at most $2$ points.
  \item $G$ bounds a properly embedded disk $\Delta_G$ that intersects $R$ in exactly one point, and intersects $B$ in at most $2$ points.
\end{itemize}
Equivalently, one needs at most two slides of $B$ and $G$ over $R$ in the special RBG-link construction for $L$.
\end{mydef}
Some small RBG-links with $R=U$ will not to be useful to construct an exotic $\nCP$.
If either of the intersection numbers $\Delta_B \cap G$ or $\Delta_G \cap B$ is strictly less than $2$, then $K_B = K_G$ by Proposition $4.11$ of \cite{zero_surg_exotic}.
If $R = U$ and $r \ge 0$, then the proof of Theorem \ref{main_thm} can be applied.
What remains can be dealt with via the following lemma.
\begin{lem}
\label{lem:small_RU_minusoneframing}
Suppose $L$ is a small RBG-link with $R=U$ and $r < 0$.
Then both $K$ and $K'$ are $(-1)$-slice in $\# |r| \CP$ with slice disks $D,D' \subset (\# |r| \CP)^\circ$ that intersects one of the exceptional spheres $\CPone$ geometrically in $3$ points and the remaining $|r|-1$ exceptional spheres $\mathbb{CP}^1$ nullhomologously.
\end{lem}
\begin{proof}
We will prove this in the case that the intersection numbers $\Delta_B \cap G$ and $\Delta_G \cap B$ are both precisely $2$.
Otherwise $K_B=K_G$ and will not be of interest.
The lemma remains true in that case and can be proved in a similar way.
We will prove the lemma by induction on the framing $r$ with base case $r = -1$ and we induct by showing the $r < -1$ case of the lemma follows from the $r + 1$ case.
This is one of those peculiar induction problems where the base case is the hard part so we will save it for last.

\begin{figure}
\centering

\subfloat[]{{
   \fontsize{10pt}{12pt}\selectfont
   \def\svgwidth{1.5in}
   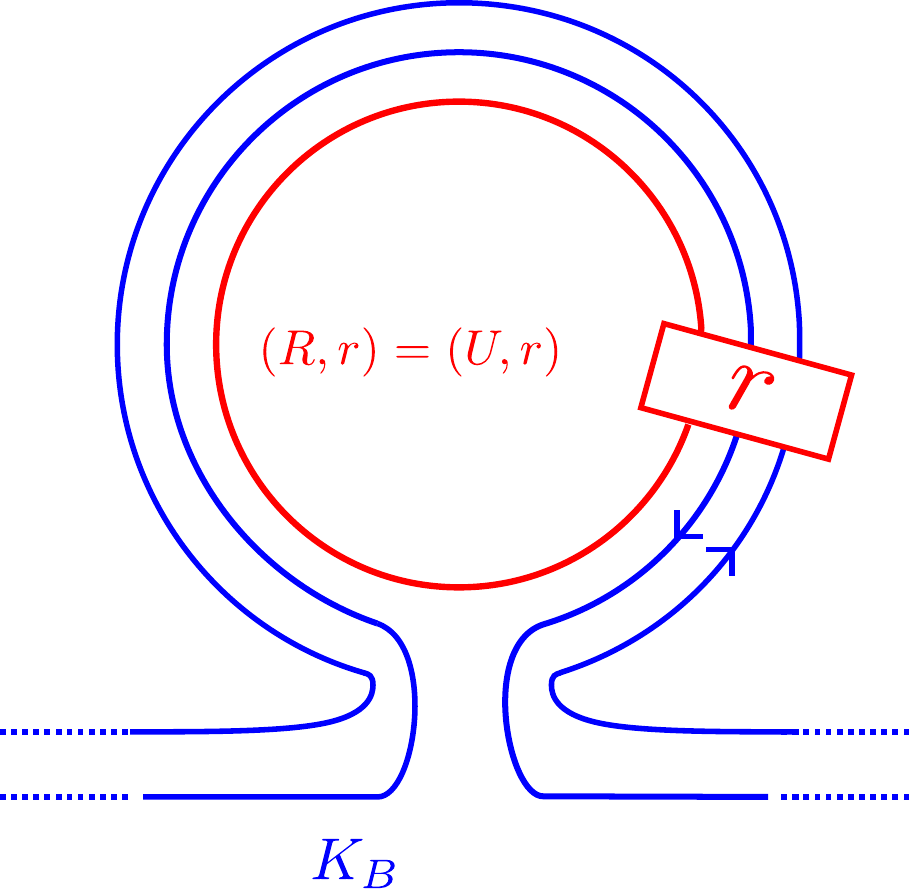
}
\label{fig:smallRBG_1}
} \ \
\subfloat[]{{
   \fontsize{10pt}{12pt}\selectfont
   \def\svgwidth{1.5in}
   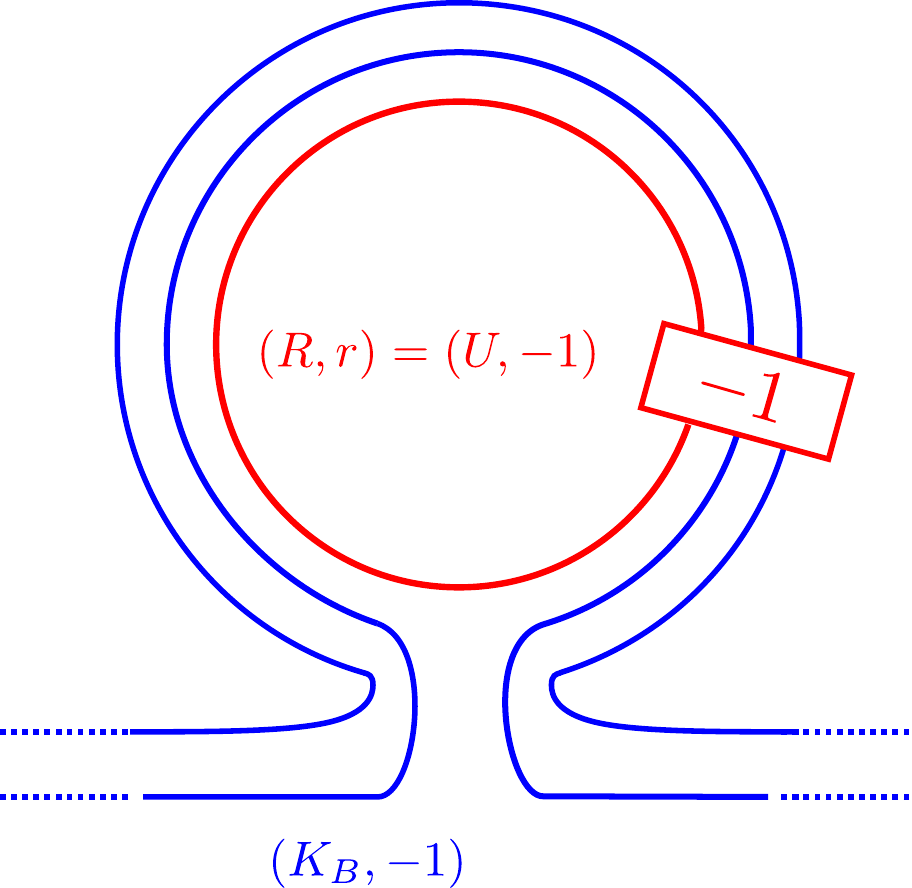
}
\label{fig:smallRBG_2}
} \ \
\subfloat[]{{
   \fontsize{10pt}{12pt}\selectfont
   \def\svgwidth{0.8in}
   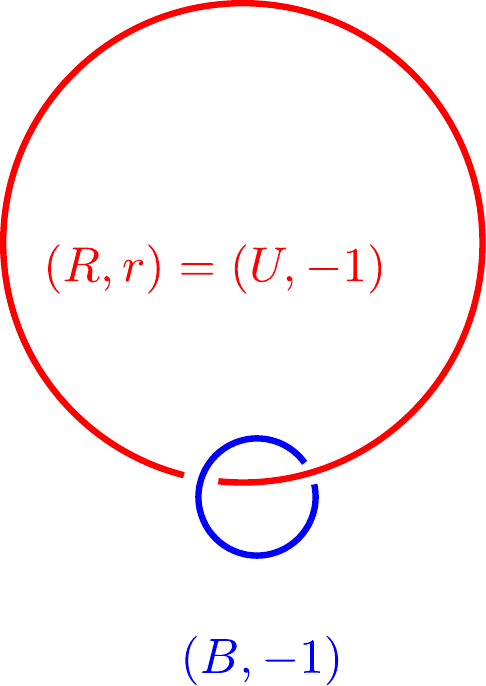
   \vspace{.20in}
}
\label{fig:smallRBG_3}
}
\caption{Constructing a slice disk for $(K_B,-1)$ in $\# |r| \CP$}
\label{fig:smallRBG_all}
\end{figure}
To prove the inductive step, let $L$ be a small RBG-link with $R = U$ and $r < -1$.
According to a linking matrix calculation at the start of Section $4.1$ of \cite{zero_surg_exotic}, $B$ and $G$ in a special RBG-link have linking number $\ell = 0$ or have $r \ell = 2$.
Since $\ell$ is just $\Delta_B \cap G = \Delta_G \cap B = 2$ counted with sign, we must have $\ell = 0$ when $r < -1$.
This $\ell$ also counts the number of slides in the slam dunk with sign.
Therefore the two strands of $B$ we slid in the slam dunk are running in opposite directions along $R$ and through the $r$ twist box as shown in Figure \ref{fig:smallRBG_1}.
Let $L^*$ be the special RBG-link obtained from $L$ by increasing $r$ to $r+1$.
The knot $K_B^*$ arising from $L^*$ differs from $K_B$ by having an $r+1$ twist box in Figure \ref{fig:smallRBG_1}.
Therefore, $K_B$ can be obtained from $K_B^*$ by a negative nullhomologous twist through the two strands running through the twist box.
If $(K_B^*,-1)$ bounds the disk $D^*$ in $\# |r+1| \CP$ intersecting a $\CPone$ in $3$ points, then we can construct such a disk $D$ for $(K_B,-1)$ as in Lemma \ref{lem:frame_CP_cob}.
The disk $D$ in $\# |r| \CP$ is simply $D^*$ after attaching a $(+1)$-framed $2$-handle to $(\# |r+1| \CP)^\circ$ along an unknot surrounding the twist box of $\partial D^* = K_B^*$.

To prove the $r=-1$ base case, let $(R,r) \cup (K_B,-1)$ be the framed link depicted in Figure \ref{fig:smallRBG_2}.
This link is obtained from $L = (R,r) \cup (B,0) \cup (G,0)$ by sliding $(B,0)$ over $(R,r)$ to get $(K_B,0)$, removing $(G,0)$, and decreasing the framing of $K_B$ to $-1$.
Identify $\CPbar -\inter (B^4 \sqcup B^4)$ as $S^3 \times I$ with a $2$-handle attached to $S^3 \times 1$ along $(R,r) = (U,-1)$.
Use this to define the cobordism $E$ in $\CPbar$ as $K_B \times I$ under this identification.
Next we slide $E$ over the $2$-handle attached to $(R,r)$ and we will represent this by slides of $(K_B,-1)$ in $(R,r) \cup (K_B,-1)$.
Starting with Figure \ref{fig:smallRBG_2}, reverse the two slides from the slam dunk to get $(R,r) \cup (B,-1)$ as in Figure \ref{fig:smallRBG_3}.
We can slide $(B,-1)$ off $(R,r)$ turning it into a zero framed unknot $(U,0)$ disjoint from $(R,r)$.
We turn $E$ upside down and cap off $(U,0)$ with a disk in $B^4$ to get a slice disk $D$ for $(K_B,-1)$ in $\CP$.

It remains to check that $D$ intersects $\CPone$ in three points.
Observe that $E$ was slid three times over the $2$-handle attached to $(R,r)$ and now $E$ has three intersections with the cocore $C$ of this $2$-handle.
$\partial C = \mu_R$ is disjoint from the $(U,0)$ boundary component of $E$.
After turning $E$ and $C$ upside down, they can be capped off without adding new intersections.
The capped off $C$ is a copy of $\CPone$ that intersects $D$ in three points.
\end{proof}
We have a $(-1)$-framed slice disk in $\# |r| \CP$ when $r < 0$ and Conjecture \ref{conj:PF_is_1_bound} would immediately tell us that $s(K') \ge 0$.
However, that conjecture is currently unconfirmed and this disk is non-trivial in homology, so we can't appeal to Lemma \ref{lem:HSlice_s_inv}.
Despite this, the approach used in \cite{gener_s_inv} to prove Lemma \ref{lem:HSlice_s_inv} will be insightful.
Let $\Sigma$ be a smooth, properly embedded, nullhomologous surface in $(\nCPbar)^\circ$ with $\partial \Sigma = K$.
First remove neighborhoods $\nu(\CPonebar) = (\CPbar)^\circ$ tubed together leaving $S^3 \times I$.
By taking these neighborhoods $\nu(\CPonebar)$ small enough, one can ensure that $\Sigma$ meets $\partial \nu(\CPonebar)$ in some link $F_{p,p}(1)$.
What remains of $\Sigma$ is a cobordism $C$ in $S^3 \times I$ from a disjoint union of $F_{p,p}(1)$ to $K$.
Now one needs to complete the difficult task of calculating $s(F_{p,p}(1))$ for this infinite family.
Once this is done, we apply our understanding of the $s$-invariant for a cobordism $C$ in $S^3 \times I$ and get constraints on $s(K)$.
By keeping careful track of the intersection data in Lemma \ref{lem:small_RU_minusoneframing}, we will be able to proceed in a similar manner.
\begin{cor}
\label{cor:small_RU_sign_of_s}
Suppose $L$ is a small RBG-link with $R=U$ and $r < 0$.
Then $s(K),s(K') \ge 0$.
\end{cor}
\begin{proof}
Let $D \subset ( \# |r| \CP)^\circ$ be the slice disk for $(K,-1)$ from Lemma \ref{lem:small_RU_minusoneframing} that intersects an exceptional sphere $\CPone$ in three points.
Delete a suitably small neighborhood $\nu(\CPone)$ of this $\CPone$ such that $D$ meets $\partial \nu(\CPone)$ in the link $-F_{2,1}(1)$.
The link $F_{2,1}(1)$ is obtained by adding a positive twist through three parallel unknots with one oriented in the opposite direction of the other two (see section $9.2$ of \cite{gener_s_inv} for details).
What remains of $D$ is a nullhomologous cobordism $C$ in $\# (|r| -1) \CP$ from $-F_{2,1}(1)$ to $K$ and we can apply Lemma \ref{lem:HSlice_s_inv}.
\[s(K) - s(-F_{2,1}(1)) \ge \chi(C) =-2\]
Tucked away before Proposition $9.9$ of \cite{gener_s_inv}, they note $s(F_{2,1}(1)) = -2$ and therefore $s(K) \ge 0$.
\end{proof}
This guarantees the proof of Theorem \ref{main_thm} will be viable for every special RBG-link $L(a,b,c,d,e,f)$ and allows us to prove the following theorem.
\begin{thm}
Let $L$ be a small RBG-link with $R = U$ and associated zero surgery homeomorphism $\phi:S^3_0(K) \rightarrow S^3_0(K')$.
\begin{enumerate}
  \item If $K$ is $H$-slice in some $\nCP$, then $s(K') \ge 0$.
  \item If $K$ is $H$-slice in some $\nCPbar$, then $s(K') \le 0$.
  \item If $K$ is slice (or more generally, biprojectively $H$-slice), then $s(K')=0$.
\end{enumerate}
In particular, this applies to any special RBG-link $\MPfam$ from the Manolescu-Piccirillo family.
\label{thm:small_RBG}
\end{thm}
\begin{proof}
It suffices to prove the first statement because the first two statements are equivalent and combining them gives the last statement.
If $r < 0$, then $s(K') \ge 0$ by Corollary \ref{cor:small_RU_sign_of_s} and for the remaining $r \ge 0$, we proceed as in the proof of Theorem \ref{main_thm}.
\end{proof}
This means that Manolescu and Piccirillo's $\MPfam$ are not suitable for finding an exotic $\nCP$ using the $s$-invariant.
Let us emphasize that this still leaves open that some $\MPfam$ could be used to construct an exotic $\nCP$.
We could still have $K$ $H$-slice in some $\nCP$ while $K'$ is not.
The above theorem shows that $s(K')$ can not obstruct $H$-sliceness of $K'$ in $\nCP$ and detect exoticness.
It would be interesting if one could show that this does not occur and generalize Theorem \ref{thm:small_RBG} on the level $H$-sliceness in $\nCP$ without reference to a particular concordance invariant.
\begin{prob}
Let $\zsg$ be a zero surgery homeomorphism arising from some $L(a,b,c,d,e,f)$.
Show that $K$ is $H$-slice in some $\nCP$ if and only if $K'$ is $H$-slice in some (ideally the same) $\nCP$.
\end{prob}
\subsection{Generalizing to Other Special RBG-links}
\label{sec:eff_range}
Now that we've ruled out using $\MPfam$ with the $s$-invariant to construct an exotic $\nCP$, we are naturally led to consider other zero surgery homeomorphisms.
One would not want to run into the same issues as $\MPfam$, so we will explain how Theorem \ref{thm:small_RBG} can be generalized to other special RBG-links.
We hope that in doing so that future attempts to construct an exotic $\nCP$ can avoid having our techniques be applicable.
For the remainder of this subsection, let $(K,K')$ be a pair of knots coming from a special RBG-link $L = (R,r) \cup (B,0) \cup (G,0)$.
Our goal is to find conditions on $(R,r)$ that allow us to use $H$-sliceness of $K$ in $\nCP$ to infer properties of $s(K')$.

Key to Theorem \ref{thm:small_RBG} was to prove Corollary \ref{cor:small_RU_sign_of_s} to get control over $s(K')$ when $r < 0$.
This was done by analyzing particular slice disks in $\# |r| \CP$ constructed in Lemma \ref{lem:small_RU_minusoneframing}.
Corollary \ref{cor:HtpyX_stable_diff} and Lemma \ref{lem:arb_zero_surg_hslice} of Subsection \ref{sec:obHslice} produce nullhomologous trace embeddings using a zero surgery homeomorphism given a slice condition on $(R,r)$.
The following is in the same spirit, but now we construct a \textit{framed trace embedding}.
\begin{lem}
\label{lem:proj_framing_from_surgery}
If $(R,r+1)$ is slice in some closed $4$-manifold $W$, then $K$ and $K'$ are $(-1)$-slice in $W \# \CP$
\end{lem}
\begin{proof}
Let $(R,r) \cup (K_B,-1)$ be the framed link obtained from $L = (R,r) \cup (B,0) \cup (G,0)$ by sliding $(B,0)$ over $(R,r)$ to turn it into $(K_B,0)$, removing $(G,0)$, and decreasing the framing of $K_B$ to $-1$.
Take this link to be a Kirby diagram of a $4$-manifold $Z$ and note that $X_{-1}(K_B)$ clearly embeds in $Z$.
Reverse the slides of the slam dunk turns $(R,r) \cup (K_B,-1)$ into $(R,r) \cup (B,-1)$.
Then slide $(R,r)$ off $(B,-1)$ to get $(R,r+1) \sqcup (U,-1)$.
These slides induce a diffeomorphism of $Z$ with $ X_{r+1}(R) \# \CPbar$.
If $(R,r+1)$ is slice in $W$, then $-X_{r+1}(R)$ embeds in $W$ by the framed trace embedding lemma.
Then $-Z = (-X_{r+1}(R)) \# \CP$ embeds in $W \# \CP$ and so does $-X_{-1}(K_B) \subset -Z$, hence $(K_B,-1)$ is slice in $W \# \CP$.
\end{proof}
Observe that these are roughly the same slides used in the $r=-1$ case of Lemma \ref{lem:small_RU_minusoneframing} and these two proofs should be thought of as essentially the same.
The above proof is much simpler because we used trace embeddings instead of directly constructing the slice disk.
That was necessary in Lemma \ref{lem:small_RU_minusoneframing} because we had to keep careful track of intersection data to avoid Conjecture \ref{conj:PF_is_1_bound}.
That conjecture asserted that if a knot $K$ is $(-1)$-slice in some $\nCP$, then $s(K) \ge 0$.
Now we will just assume Conjecture \ref{conj:PF_is_1_bound} and apply it to the $(-1)$-slicing of $K$ in $\# (n+1) \CP$ given by the above lemma with $W= \nCP$.
Here we'll state this in terms of the projective slice framing $\PF_-(R)$ from Subsection \ref{sec:proj_slice_framing}.
If $r < \PF_-(R)$, then $r+1 \le \PF_-(R)$ and $(R,r+1)$ is slice in some $\nCP$ according to Corollary \ref{cor:atleast_PF}.
\begin{cor}
If $r < \PF_-(R)$, then $K$ and $K'$ are both $(-1)$-slice in some $\nCP$.
If Conjecture \ref{conj:PF_is_1_bound} is true, then $s(K),s(K') \ge 0$.
\qed
\label{cor:PF_0_or_1}
\end{cor}
We apply this Corollary \ref{cor:PF_0_or_1} with an arbitrary special RBG-link in the same way we used Corollary \ref{cor:small_RU_sign_of_s} for the Manolescu-Piccirillo $\MPfam$.
This proves the following which characterizes when our methods can be applied to a zero surgery homeomorphism.
\begin{thm}
Let $\phi: S^3_0(K) \rightarrow S^3_0(K')$ be a zero surgery homeomorphism arising from a special RBG-link $L = (R,r) \cup (B,0) \cup (G,0)$ and assume Conjecture \ref{conj:PF_is_1_bound} is true.
\begin{enumerate}[itemsep=0em]
    \item Suppose $r < \PF_-(R)$ or $r \ge \PF_+(R)$.
    If $K$ is $H$-slice in some $\nCP$, then $s(K') \ge 0$.
    \item Suppose $r \le \PF_-(R)$ or $r > \PF_+(R)$.
    If $K$ is $H$-slice in some $\nCPbar$, then $s(K') \le 0$.
    \item Suppose $r < \PF_-(R)$, $r > \PF_+(R)$, or $R$ is biprojectively $H$-slice with any $r \in \Z$.
    If $K$ is slice (or biprojectively $H$-slice), then $s(K') = 0$.
\end{enumerate}
Moreover, if $R$ is biprojectively $H$-slice, then the conditions on $(R,r)$ automatically hold.
\label{thm:full_adj_version}
\end{thm}
\begin{proof}
The $r < \PF_-(R)$ part of the first statement is Corollary \ref{cor:PF_0_or_1}.
If $r \ge \PF_+(R)$, then $(R,r)$ is slice in some $\# k \CPbar$ by Corollary \ref{cor:atleast_PF}.
If we also have $K$ is $H$-slice in some $\nCP$, then we can conclude $K'$ is $H$-slice in $\# (n+k) \CP$ by Corollary \ref{cor:HtpyX_stable_diff} and $s(K') \ge 0$ by Lemma \ref{lem:HSlice_s_inv}.
For the final statement, the first two statements simultaneously apply when $r < \PF_-(R)$ and $r > \PF_+(R)$.
If $R$ is biprojectively $H$-slice, then $\PF_+(R) = \PF_-(R) = 0$ by Lemma \ref{lem:proj_Hslice} and the first two statements simultaneously apply for any $r \in \Z$. 
\end{proof}
The special RBG-links $\MPfam$ all have biprojectively $H$-slice $R=U$ and this might explain why Manolescu and Piccirillo were unsuccessful in their search.
The appeal of the $s$-invariant for detecting an exotic $\nCP$ was that it was not clear if the obstruction should apply in an exotic $\nCP$.
This theorem sometime recovers this obstruction if we only understand the $s$-invariant in the standard $\nCP$.
Furthermore, this theorem could apply to other concordance invariants that shares the properties of the $s$-invariant we used (e.g. any that satisfies an adjunction inequality in $\nCP$).
This is troubling for the prospect of constructing an exotic $\nCP$ from zero surgery homeomorphisms.
Once we understand a concordance invariant in the standard $\nCP$, it can often be enough to rule out using it to construct an exotic $\nCP$ with zero surgery homeomorphisms.

However, this theorem does not immediately apply to all zero surgery homeomorphisms.
This leaves open the possibility that some zero surgery homeomorphism could be used to construct an exotic $\nCP$.
We will construct an infinite family of special RBG-links for which our methods do not apply.
Our special RBG-links will come in the form of a special RBG-link $L$ with a local connected sum by some knot $J$ to $R$.
Call this new special RBG-link $L[J]$ and the resulting knots $K_B[J],K_G[J]$.

Dunfield and Gong used topological slice obstructions to show that $K_6,\dots,K_{21}$ are not slice.
By viewing $K_B[J],K_G[J]$ as satellites knots, we get some control over the topological sliceness of the resulting knots.
Examining Figure \ref{fig:slamdunk_all}, we see that $K_B[J]$ and $K_G[J]$ are both satellites $P_B(J)$ and $P_G(J)$ of $J$.
These will be patterns $P_B$ and $P_G$ such that $K_B = P_B(U)$ and $K_G = P_G(U)$.
These patterns have winding number equal to the linking number $\ell$ of $B$ and $G$.
We will take $L$ to be one of the special RBG-links $L_i$ associated to the five pairs $\{(K_i,K_{i}')\}_{ i=1,\dots,5 }$.
Denote the knots resulting from $L_i[J]$ by $K_i[J]$ and $K_i'[J]$.
These $L_i$ all have $\ell = 0$ and so the associated satellite patterns have winding number zero.
$K_i[J]$ and $K_i'[J]$ will then have the same Alexander polynomials as $K_i$ and $K_i'$ by Seifert's formula for the Alexander polynomial of a satellite \cite{seifert_alex_poly}.
In particular, $K_i[J]$ and $K_i'[J]$ will have trivial Alexander polynomial and will be topologically slice by Freedman \cite{Freedman}.

Let $r_i$ denote the framing of $R$ in $L_i$ which will be equal to $1$, $2$, or $3$.
To construct an exotic $\nCP$ from $L_i[J]$, Theorem \ref{thm:full_adj_version} suggests we should have $J$ not biprojectively $H$-slice and have $\PF_-(J) \le r_i < \PF_+(J)$.
Note that the condition $\PF_-(J) \le r_i$ holds automatically so we only need to check that $\PF_+(J) > r_i$.
Such $J$ are in abundance as any $J$ with $\tau(J) \ge 1$, such as the right hand trefoil, will suffice due to Corollary \ref{cor:PF_bounds_tau}.

These $L_i[J]$ give an infinite family of special RBG-links which are not susceptible to topological slice obstructions or the methods of this paper.
One could potentially apply the methodology of Manolescu and Piccirillo to these families.
We do not propose these as a serious attempt at constructing an exotic $\nCP$.
It seems that going from $(K,K')$ to $(K[J],K'[J])$ would increase slice genus since the resulting knots are more complicated.
Instead, we propose these as a setting to study how to relate the $s$-invariants and $H$-sliceness of knots with homeomorphic zero surgeries.
\begin{prob}
Let $r_i < \PF_+(J)$, relate $H$-sliceness of $K_i[J],K_i'[J]$ in $\nCP$ and their $s$-invariants to each other.
In particular, show that if one of these knots is $H$-slice in some $\nCP$, then the other has non-negative $s$-invariant.
\end{prob}

\label{sec:annulus_twist_htpy_S4}
\section{Annulus Twist Homotopy Spheres}
Manolescu and Piccirillo constructed homotopy $4$-spheres $Z_n = E(D) \cup_{\phi_n} -X_0(J_n)$ from annulus twisting a ribbon knot $J_0$.
In this section, we show that these $Z_n$ are all standard by drawing them upside down as $-Z_n = X_0(J_n) \cup_{\phi_n} -E(D)$.
This proof was motivated by a desire to understand and visualize the trace embedding of $X_0(J_n)$ in $-Z_n$ as a Kirby diagram.
We will first need to explain how this works for a ribbon knot and its associated trace embedding into $S^4$.
\subsection{A Kirby Diagram of the Trace Embedding Lemma}
\label{sec:ribbon_kirby}
A ribbon disk is a smoothly, properly embedded disk $D$ in $B^4$ such that the height function on $B^4$ restricted to $D$ has no index two critical points.
A knot $K$ is called a ribbon knot if it bounds a ribbon disk.
Similarly, an $n$ component link $\mathcal{L}$ is called a ribbon link if it bounds a collection $\mathcal{D} = D_1 \cup \dots \cup D_n$ of $n$ disjoint ribbon disks called a ribbon disk link.
A ribbon disk is typically described by a ribbon diagram.
This is an $n$ component unlink $\mathcal{U}$ together with a collection of $n-1$ ribbon bands.
A ribbon band is an embedded $I \times I$ attached to $\mathcal{U}$ along $(\partial I) \times I$ (respecting orientations).
These ribbon bands $I \times I$ must intersect the disks that bound $\mathcal{U}$ as some $a \times I$.
The knot $6_1$ is ribbon and can be described by the ribbon diagram in Figure \ref{fig:Ribbon_diag}.
\begin{figure}
\centering
\subfloat[]{
    \def\svgwidth{2in}
   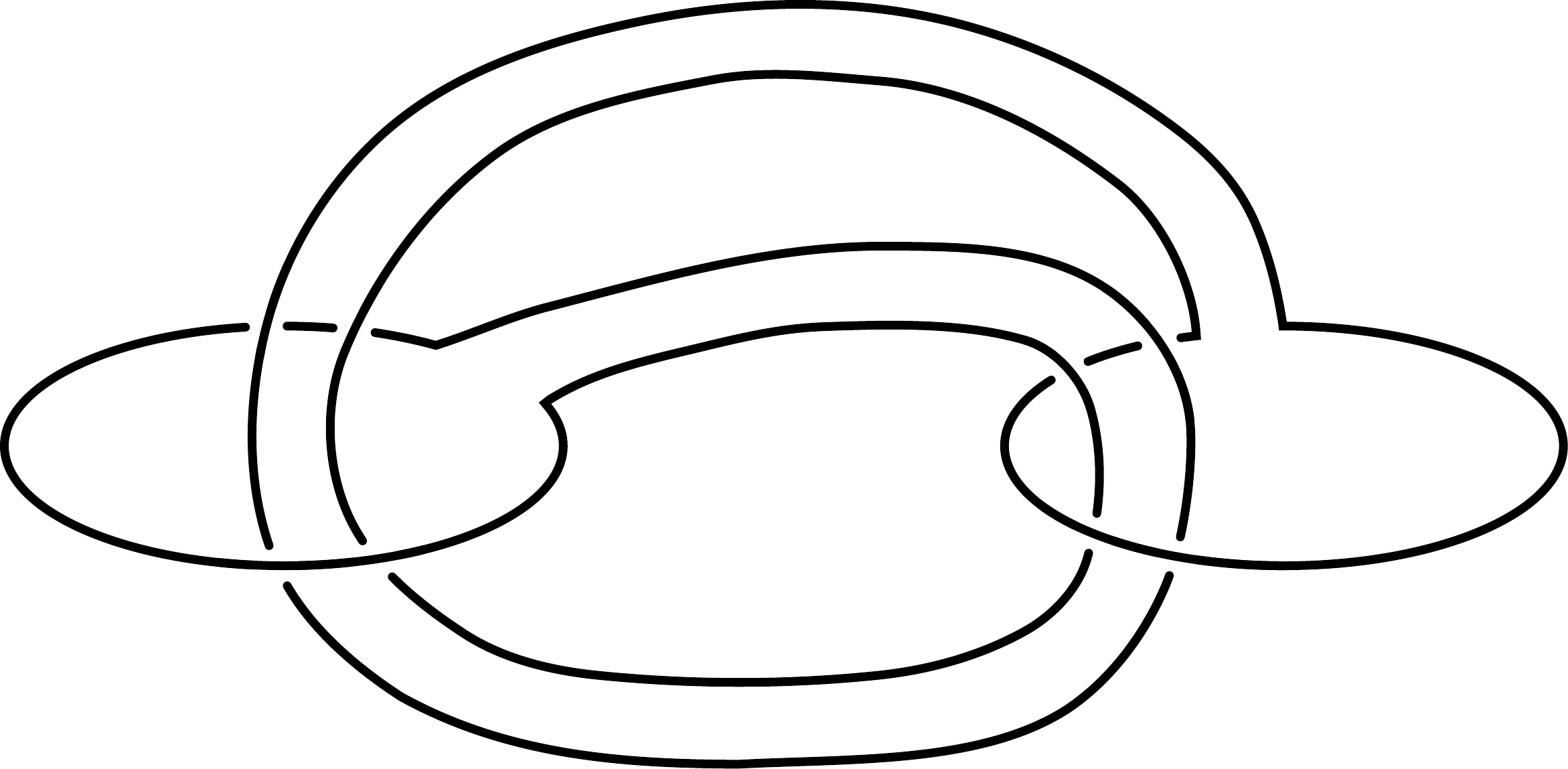
       \label{fig:Ribbon_diag}
}
\subfloat[]{
   \fontsize{10pt}{12pt}\selectfont
   \def\svgwidth{2in}
   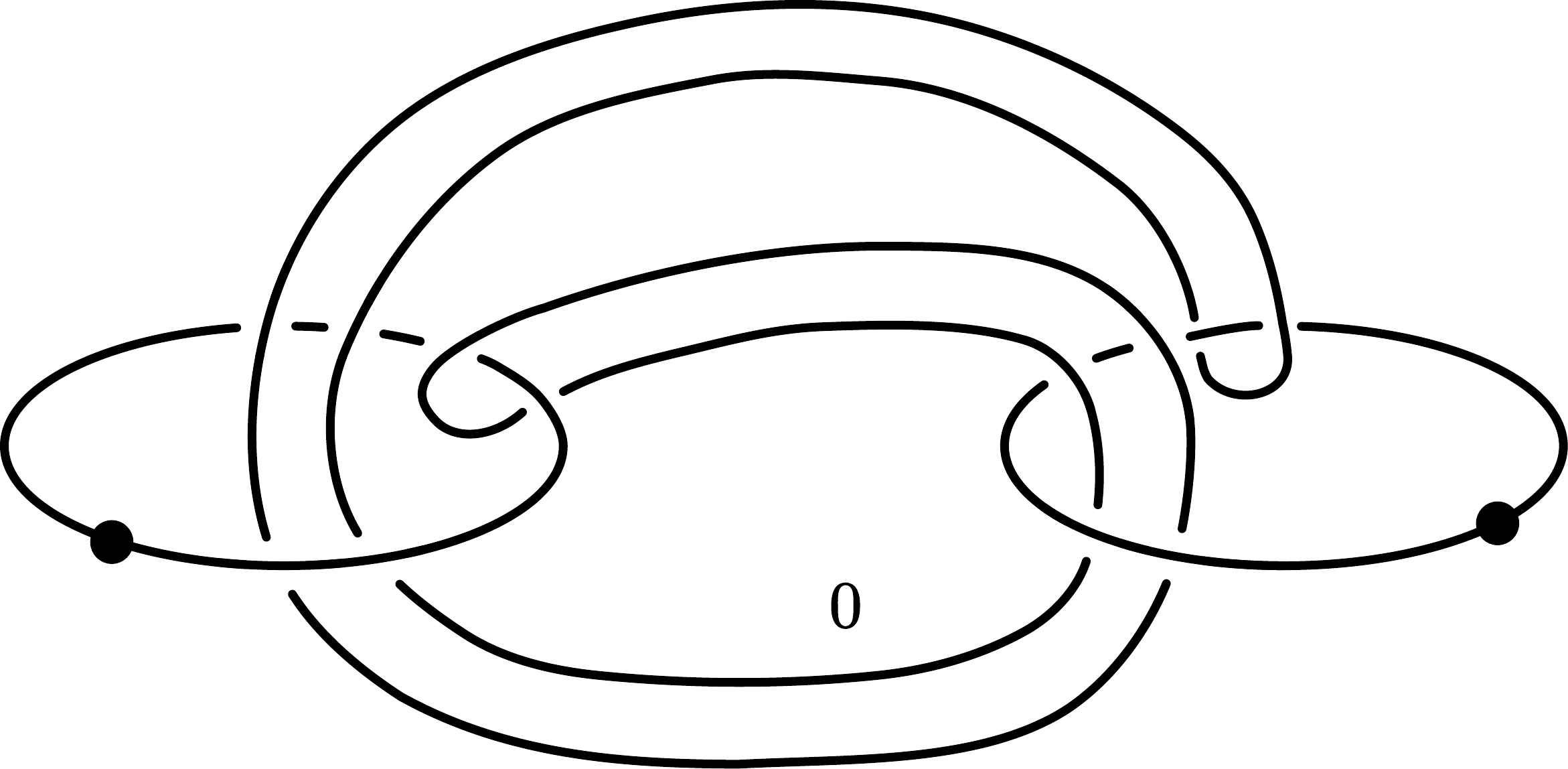
   \label{fig:Ribbon_Exterior_Kirby}
}

\label{fig:ribbon_def}
\caption{A ribbon diagram for $6_1$ and a Kirby diagram of the corresponding the ribbon disk exterior}
\end{figure}

Such a ribbon diagram for a knot $K$ defines a ribbon disk $D$ bounding $K$.
Each component of the unlink $\mathcal{U}$ becomes an index zero critical point of $D$ and each ribbon band becomes an index one critical point of $D$.
We can draw a Kirby diagram of the exterior $E(D) = B^4 - \nu(D)$ of $D$ from its Ribbon diagram with the algorithm presented in Section $6.2$ of \cite{4_Mfld_Kirby}.
Using Figure \ref{fig:Ribbon_diag}, we draw a Kirby diagram for the exterior of a ribbon disk of $6_1$ in Figure \ref{fig:Ribbon_Exterior_Kirby}.
The index zero critical points of $D$ become $1$-handles which we draw by putting a dot on each component of $\mathcal{U}$.
The index one critical points of $D$ become zero framed $2$-handles which follow the boundary of the corresponding ribbon like in Figure \ref{fig:Ribbon_Exterior_Kirby}.

To fill in $E(D)$ in this diagram, we attach a zero framed $2$-handle to a meridian of a dotted circle.
We can then cancel this pair, the rest of the diagram ``unravels'', and the remaining handles cancel.
This leaves $B^4$ and capping off with a $4$-handle gives $S^4$.
This gives a decomposition of $S^4$ as $E(D)$ with a $2$-handle and $4$-handle attached.
These additional handles represent an embedded $-X_0(K)$ and this is the same embedding as the classical trace embedding lemma.
\begin{lem}[Trace Embedding Lemma]
$K$ is slice if and only if $X_0(K)$ (equivalently $-X_0(K)$) embeds in $S^4$.
\end{lem}
While such a diagram gives the same embedding of $-X_0(K)$, we don't clearly see the embedded trace.
It is represented as a $2$-handle attached to a meridian of a dotted circle and a $4$-handle.
We would rather see the trace embedding as a $2$-handle attached to $K$ in our diagram.
To do this, we turn this decomposition upside down as $X_0(K) \cup -E(D)$ and draw the corresponding Kirby diagram.
One could try to use the standard method to turn a Kirby diagram upside down as in Section $5.5$ of \cite{4_Mfld_Kirby}.
The difficulty with that method is that turning $E(D)$ upside down directly can result in a messy Kirby diagram.
The method we will describe will result in simpler diagrams that can be read off directly from a ribbon diagram of $D$.
To do this, we will upgrade $K$ and $D$ to a ribbon link $\mathcal{L} = K \cup L_1 \cup \dots \cup L_{n-1}$ and ribbon disk link $\mathcal{D} = D \cup D_1 \cup \dots \cup D_{n-1}$.
This ribbon disk link will have exterior $E(\mathcal{D})$ consisting of only a $0$-handle and $n$ $1$-handles which can be turned upside down immediately.
We will give pictures of how to do this for $K = 6_1$ and $D$ its standard ribbon disk shown in Figure \ref{fig:Ribbon_diag}.

Draw a ribbon diagram of $D$ and add a small unknot $L_i$ encircling each ribbon band to get a link $\mathcal{L} = K \cup L_1 \cup \dots \cup L_{n-1}$.
Since each $L_i$ bounds a disk $D_i$ in $S^3$ that intersect $K$ in ribbon singularities, there is a ribbon disk link $\mathcal{D} = D \cup D_1 \cup \dots \cup D_{n-1}$ for $\mathcal{L}$ where each $D_i$ has a unique index zero critical point.
To draw a Kirby diagram for $E(\mathcal{D})$, add a red dotted circle to each $L_i$ in the diagram of $E(D)$ as in Figure \ref{fig:disk_link_exterior}.
\begin{figure}
\centering
\subfloat[]{
    \def\svgwidth{2in}
   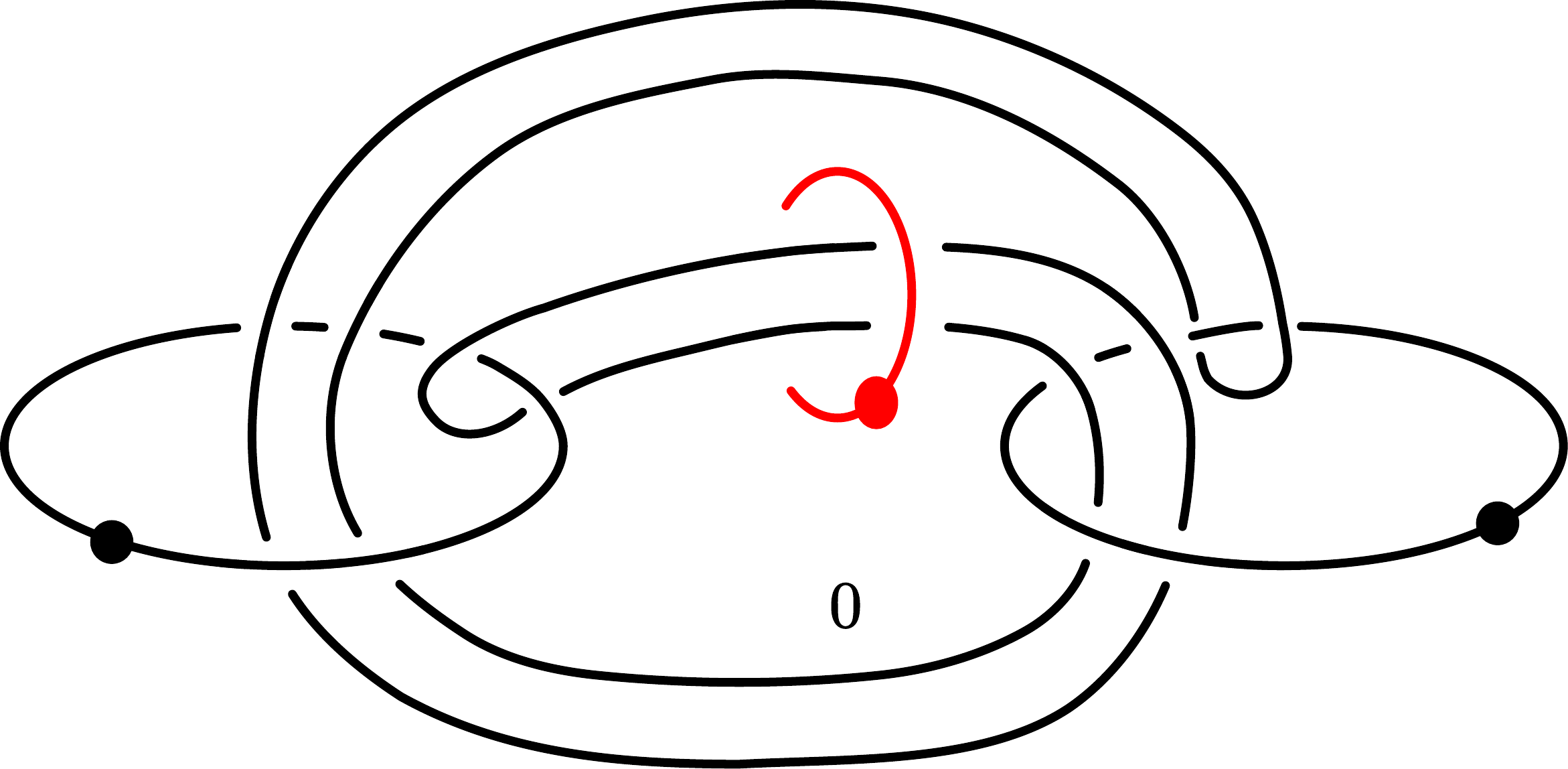
    \label{fig:disk_link_exterior}
} \ \ \ \ \ \ \
\subfloat[]{
    \def\svgwidth{2in}
   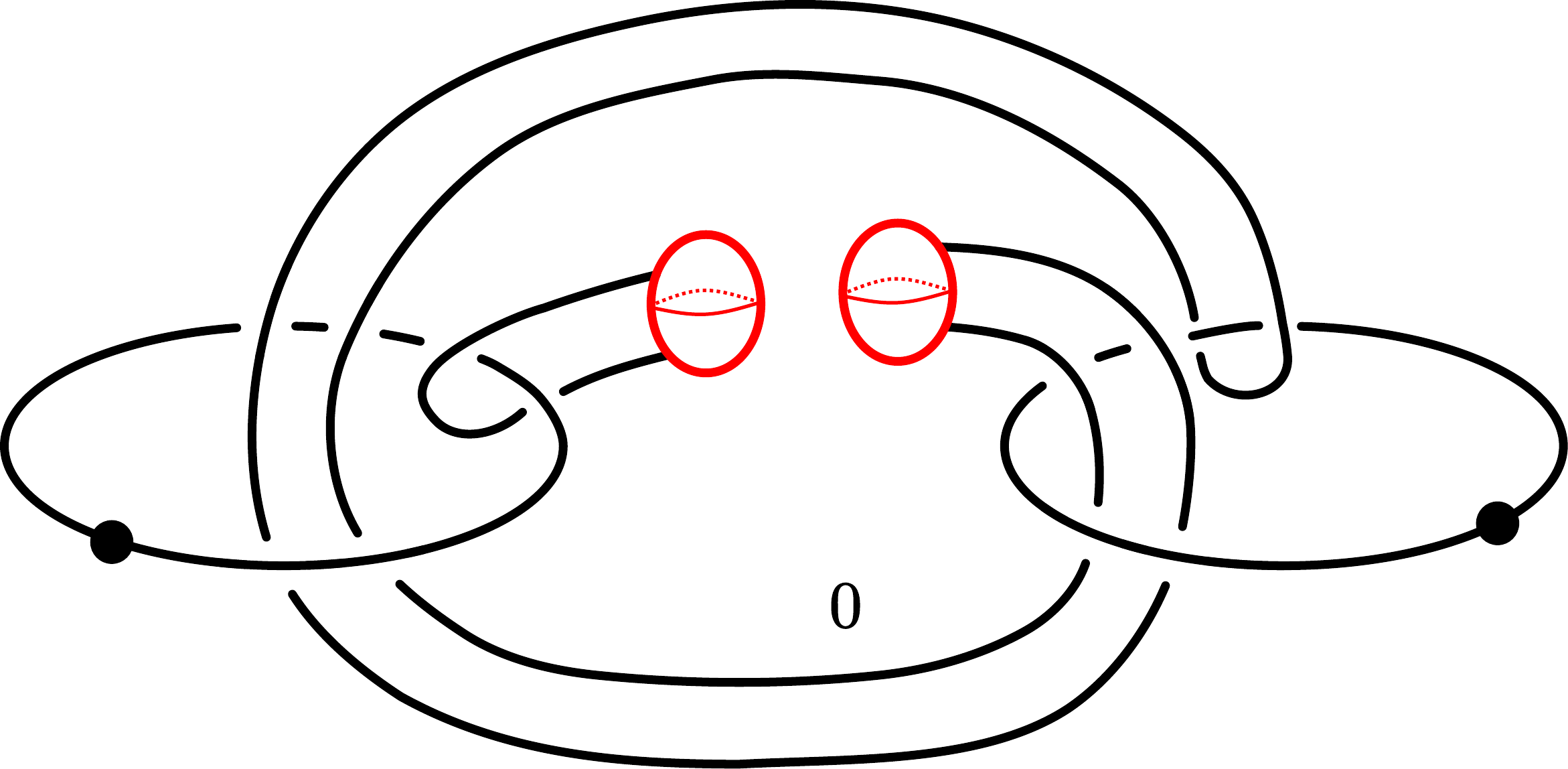
       \label{fig:exterior_balls}
   } \\

\subfloat[]{{
    \fontsize{10pt}{12pt}\selectfont
    \def\svgwidth{2in}
    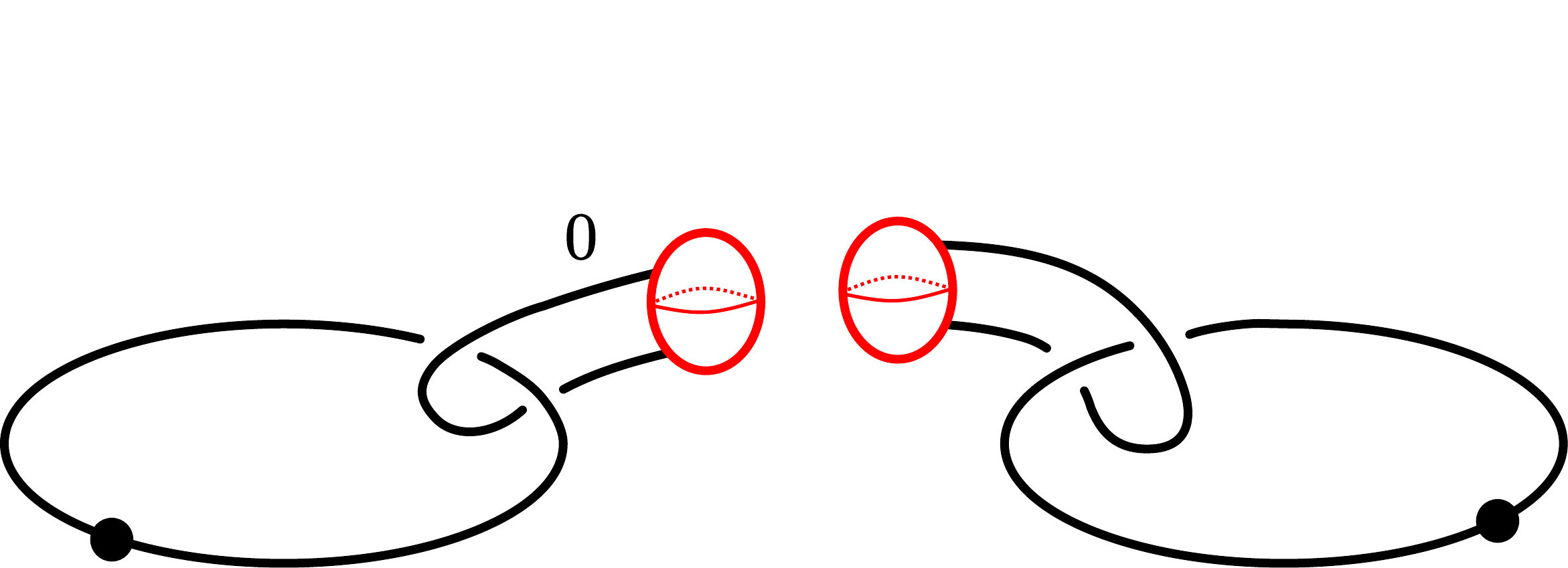
}
\label{fig:simplifying_exterior_a}
} \ \ \ \ \ \ \
\subfloat[]{{
   \fontsize{10pt}{12pt}\selectfont
   \def\svgwidth{2in}
   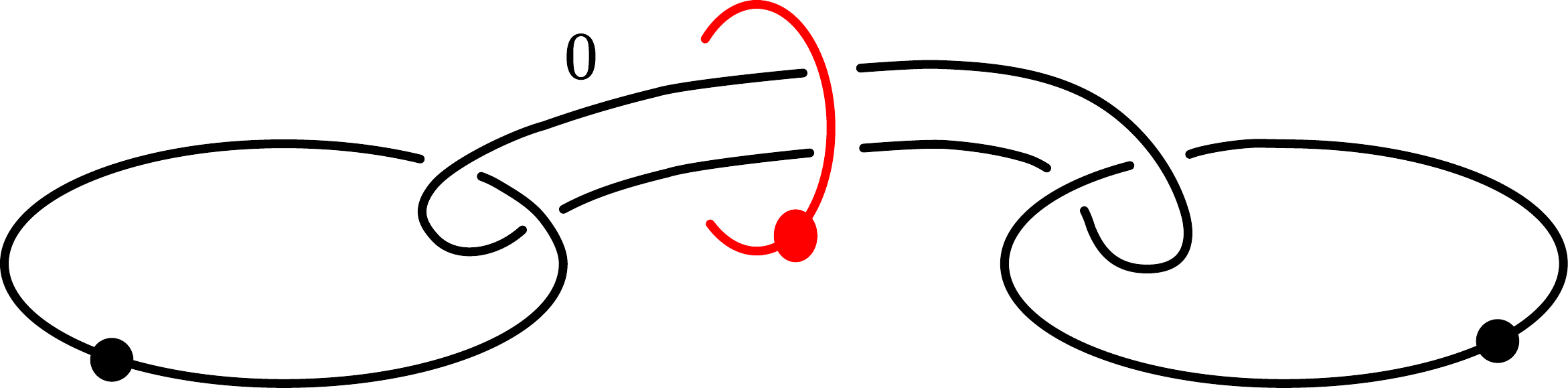
}
\label{fig:simplifying_exterior_b}
}
\caption{Simplifying $E(\mathcal{D})$}

\end{figure}
To simplify, change the red dotted circles into a pair of balls to represent $1$-handles as in Figure \ref{fig:exterior_balls}.
Think of this change in notation as doing ribbon moves on $K$ and each of these balls as a pair of arcs on the banding of $K$.
We can pull each of the red balls along the band to get Figure \ref{fig:simplifying_exterior_a}.
What is left is $n$ black dotted circles with $n-1$ $2$-handles each running through a pair of balls connecting them in consecutive pairs.
Change the balls back to dotted notation as in Figure \ref{fig:simplifying_exterior_b} and then move the red dotted circles off the rest of the diagram.
Cancel the $2$-handles leaving a single black dotted circle and $n-1$ red dotted circles.
We conclude that $E(\mathcal{D})$ admits a handle decomposition with one $0$-handle with $n$ $1$-handles attached.

Let $X_0(\mathcal{L})$ be obtained by attaching zero framed $2$-handles to $B^4$ along each component of $\mathcal{L}$.
We attach $-E(\mathcal{D})$ to $X_0(\mathcal{L})$ to get a Kirby diagram of $S^4$.
The handles of $E(\mathcal{D})$ turn upside down to become $3$ and $4$-handles which attach uniquely.
To summarize, we attach zero framed $2$-handles to a ribbon knot $K$ and unknots encircling the $n-1$ ribbon bands of $D$, then cap off with $n$ $3$-handles and a $4$-handle.
For $K=6_1$, we get the Kirby diagram in Figure \ref{fig:6_1_sphere}.
\begin{remark}
Some of the more adept practitioners of Kirby Calculus may have applied the standard method to turn a Kirby Diagram upside down and got the same diagram as Figure \ref{fig:6_1_sphere}.
However, this is an artifact of the simplicity of $6_1$ and you will not get the same diagram for a more complicated ribbon knot.
One can see what goes wrong if one tries this with the ribbon knot shown in Figure \ref{fig:annulus_ribbon} used in the following subsection.
If one uses that method to draw the homotopy spheres of interest, one gets diagrams that do not seem as amenable to simplification.
\end{remark}
\begin{figure}[!htbp]
\centering
{
    \fontsize{10pt}{12pt}\selectfont
    \def\svgwidth{2.5in}
    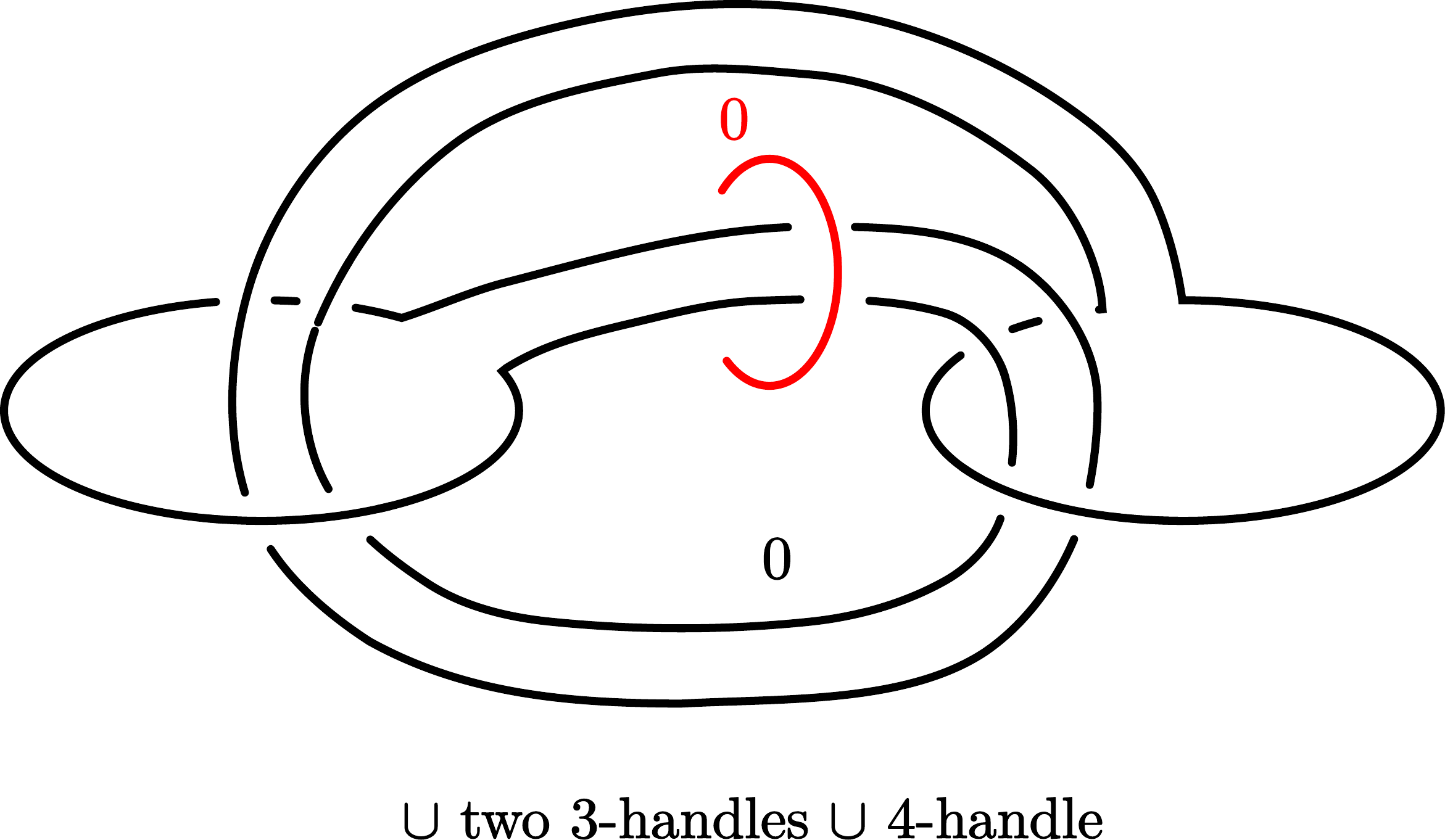
}\caption{Kirby diagram of $S^4$ as $X_0(6_1) \cup -E(D)$}
\label{fig:6_1_sphere}
\end{figure}
%
%
\subsection{Standardzing \texorpdfstring{$Z_n$}{}}
%
%
\begin{figure}
\centering
{
   \fontsize{10pt}{12pt}\selectfont
   \def\svgwidth{2.5in}
   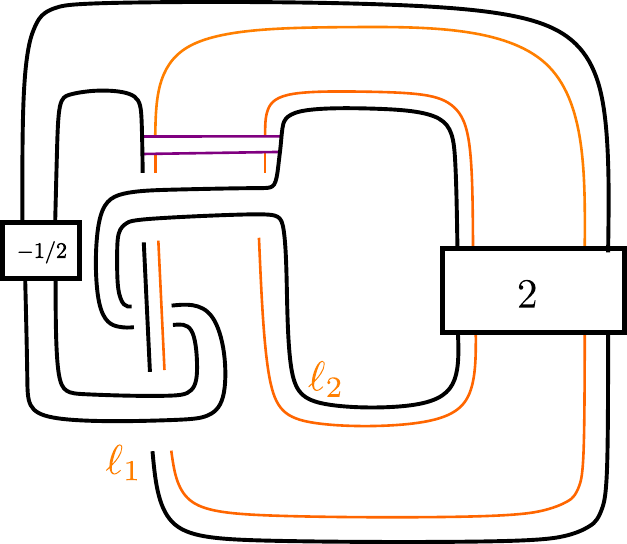
}\caption{Annulus presentation and ribbon diagram of $J_0$}
\label{fig:annulus_ribbon}
\end{figure}
Manolescu and Piccirillo constructed and drew diagrams of a family of homotopy spheres that we will call $Z_n$.
These $Z_n$ arise from a family of knots $\{ J_n \}_{n \in \Z}$ which are annulus twists on the ribbon knot $J_0 = 8_8$.
The annulus presentation and ribbon diagram of $J_0$ are depicted in Figure \ref{fig:annulus_ribbon}.
The half fractional box notation in that figure represents half twists on the two strands passing through that box.

Annulus twisting was introduced by Osoinach in his construction of infinite collections of knots that share a zero surgery \cite{ann_twist}.
The annulus presentation of $J_0$ defines a family of homeomorphisms $\phi_n : S^3_0(J_0) \rightarrow S^3_0(J_n)$.
In Figure \ref{fig:annulus_ribbon}, we see that $\ell_1 \cup \ell_2$ are the boundary of an annulus $A$.
This annulus induces orientations and framings on $\ell_1$ and $\ell_2$.
$J_0$ was obtained by banding $\ell_1 \cup \ell_2$ and so all three cobound a pair of pants.
Let $A' \subset S^3_0(J_0)$ be an annulus formed by the union of the surgery disk and the pair of pants.
Twisting along $\nu(A')$ gives a homeomorphism $S^3_0(J_0) \rightarrow S^3_{0,1/n,-1/n}(J_0,\ell_1,\ell_2)$ (where framings on $\ell_1 \cup \ell_2$ are relative to their annulus framings).
Then a twist on $\nu(A) \subset S^3$ gives a homeomorphism $S^3_{1/n,-1/n}(\ell_1,\ell_2) \rightarrow S^3$ and identifies $S^3_{0,1/n,-1/n}(J_0,\ell_1,\ell_2)$ with zero surgery on some knot $J_n \subset S^3$.
The annulus twist homeomorphisms $\phi_n : S^3_0(J_0) \rightarrow S^3_0(J_n)$ is the composition of these homeomorphisms.

$J_0$ is ribbon and bounds a disk $D$ described by the ribbon move in Figure \ref{fig:annulus_ribbon}.
We can use the annulus twist homeomorphisms to construct the homotopy spheres $Z_n = E(D) \cup_{\phi_n} -X_0(J_n)$.
This decomposition as $E(D) \cup_{\phi_n} -X_0(J_n)$ is the one used to draw the diagrams of $Z_n$ in Figure $20$ of \cite{zero_surg_exotic}.
We will use the technique from the previous subsection to draw diagrams of these $Z_n$ upside down as $X_0(J_n) \cup_{\phi_n} -E(D)$ and then show that each $Z_n$ is standard.
\begin{thm}
\label{thm:Z_n_standard}
All $Z_n$ are diffeomorphic to $S^4$.
\end{thm}
\begin{proof}
\begin{figure}
\centering
\subfloat[Diagram of $-Z_0$]{{
   \fontsize{10pt}{12pt}\selectfont
   \def\svgwidth{2in}
   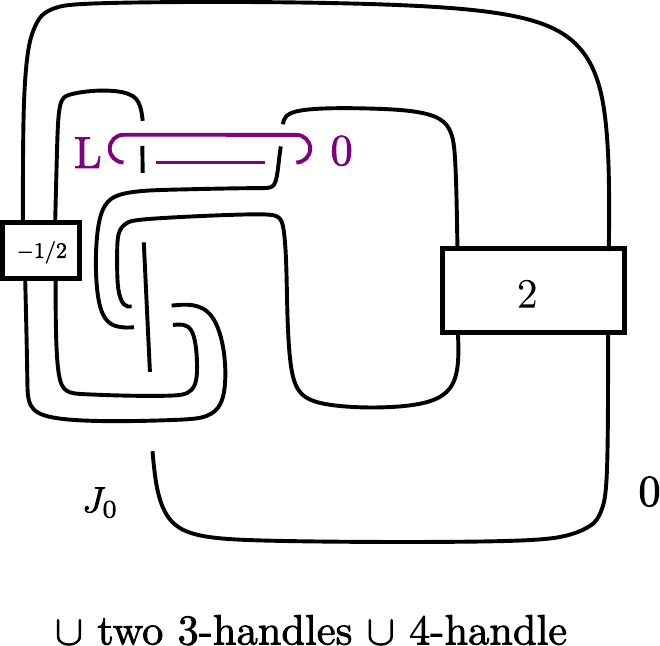
}
\label{fig:diagrams_a}
}\ \ \ \ \ \ \
\subfloat[Diagram of $-Z_n$ with an arrow indicating the upcoming slide]{{
   \fontsize{10pt}{12pt}\selectfont
   \def\svgwidth{2in}
   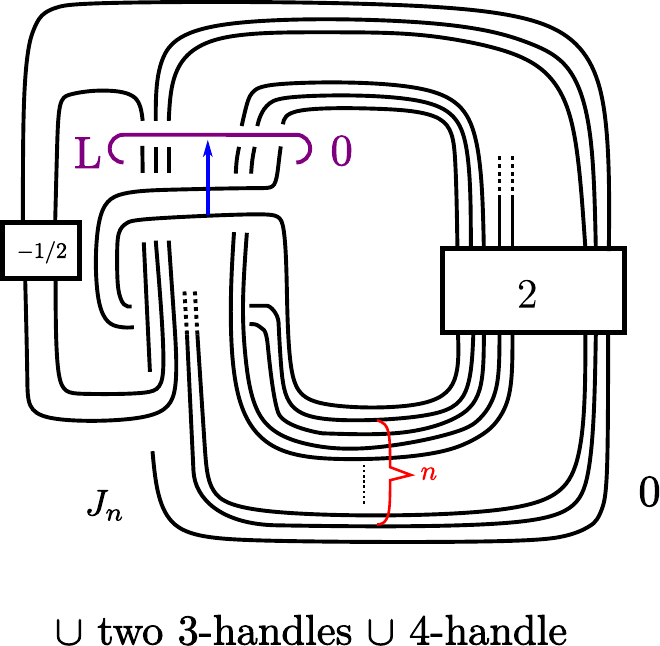
}
\label{fig:diagrams_b}
} \\

\subfloat[Result of slide, then isotope band into blue]{{
   \fontsize{10pt}{12pt}\selectfont
   \def\svgwidth{2in}
   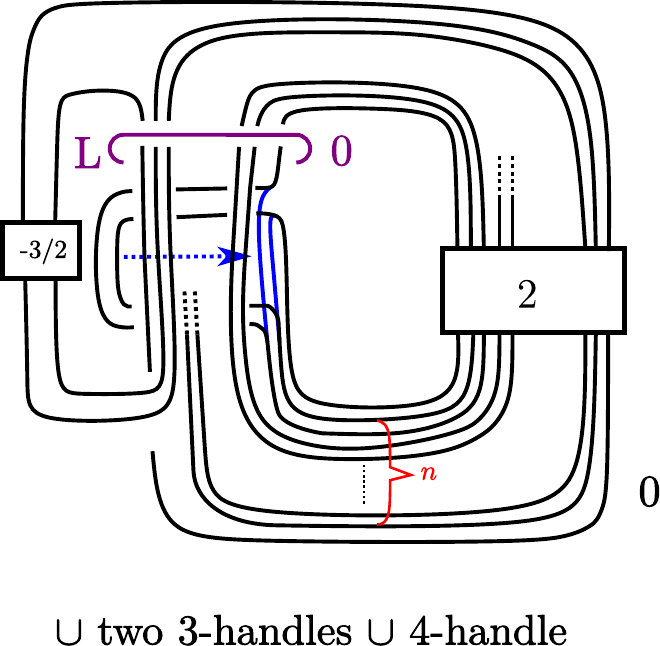
}
\label{fig:diagrams_c}
}\ \ \ \ \ \ \
\subfloat[Result of spinning the band around the annulus]{{
   \fontsize{10pt}{12pt}\selectfont
   \def\svgwidth{2in}
   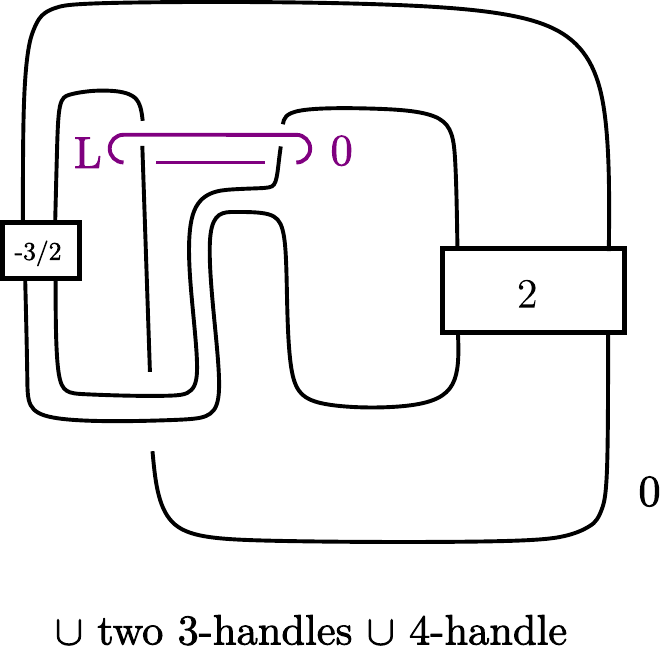
}
\label{fig:diagrams_d}
}

\caption{Standardizing $Z_n$}
\label{fig:diagrams}

\end{figure}
By definition, $Z_0 =E(D) \cup -X_0(J_0)$ is standard and we can draw $-Z_0 = X_0(J_0) \cup -E(D)$ as shown in Figure \ref{fig:diagrams_a}.
This diagram has no $1$-handles and two $2$-handles attached with zero framing to $J_0$ and the knot $L$ that surrounds the ribbon band of $J_0$.
To draw $-Z_n = X_0(J_n) \cup_{\phi_n} -E(D)$, we attach the handles of $-E(D)$ using the map $\phi_n$.
The $3$ and $4$-handles attach uniquely so we only need to keep track of $L$.
The annulus twist homeomorphism is supported in a neighborhood of $A'$ and $A$ disjoint from $L$.
Therefore, $L$ is unaffected by $\phi_n$ and we can draw a diagram of $-Z_n$ as in Figure \ref{fig:diagrams_b}.

With the desired diagrams of $-Z_n$ now in hand, we can now proceed to show that each $Z_n$ is standard.
First slide the band of $J_n$ over $L$ as shown by the blue arrow in Figure \ref{fig:diagrams_b} to get Figure \ref{fig:diagrams_c}.
This brings the band under the annulus with a twist which we absorb into the twist box on the left.
Drag this band under the annulus via the dashed blue arrow to the position shown in solid blue in Figure \ref{fig:diagrams_c}.
Spin the band around the annulus to undo the twists and get Figure \ref{fig:diagrams_d}.
We get the same diagram for all $n$, therefore all $Z_n$ must be diffeomorphic to each other and in particular $Z_0 = S^4$.
\end{proof}
\begin{cor}
$J_n$ is slice for all $n \in \Z$.
\qed
\end{cor}
This proof bears similarities to work of Akbulut and Gompf on the family $\{\Sigma_n\}_{n \in \Z}$ Cappell-Shaneson spheres \cite{CS_standard,More_Cappell}.
The resemblance is most immediate when compared to Akbulut's diagrammatic proof that each $\Sigma_n$ is standard.
Akbulut added a cancelling $2$ and $3$-handle pair to a Kirby diagram of $\Sigma_n$ and identified all $\Sigma_n$ with each other.
In particular, every $\Sigma_n$ is diffeomorphic to $\Sigma_0$ which was known to be standard by earlier work of Gompf \cite{akb_kir_standard}.
There seems to be a more opaque connection to Gompf's work following up on Akbulut.
There Gompf showed that certain torus twists don't affect the Cappell-Shaneson construction to give a mostly Kirby calculus free proof.
It would be interesting if one could think of these annulus twists in a way that recasts this proof in a similar manner.

Meier and Zupan recently gave a new proof that the Cappell-Shaneson spheres $\Sigma_n$ are standard using ideas from the theory of Generalized Property $R$ \cite{meier_zupan}.
We can give another proof that $Z_n$ is standard in a similar manner once we have the diagram shown in Figure \ref{fig:diagrams_b}.
\begin{proof}
The diagram of $-Z_n$ has no $1$-handles and $2$-handles attached to the link $J_n \cup L$ where $L$ is an unknot.
To be able to attach the two $3$-handles and the $4$-handle, zero surgery on $J_n \cup L$ must be $\#2 S^1 \times S^2$.
We can now appeal to Property $2R$ for the unknot.
\begin{lem}[Proposition 3.2 of \cite{Prop2r}]
\label{lem:unknot_2R}
The unknot has Property $2R$.
Namely, if $\mathcal{L}$ is a $2$ component framed link with an unknotted component that surgers to $\#2 S^1 \times S^2$, then there is a sequence of handle slides turning $\mathcal{L}$ into a zero framed $2$ component unlink.
\end{lem}
We can do these handle slides to $-Z_n$ and then cancel the $2$-handles with the $3$-handles to get $S^4$.
\end{proof}
This only gives \textit{existence} of handle slides that standardize $Z_n$.
We prefer the first proof where we directly see how to standardize these diagrams of $-Z_n$.
While using Property $2R$ may seem to be much slicker, it also relies on deep work of Gabai and Scharlemann \cite{scharl_reduce, Gabai_Prop_R}.
The Property $2R$ approach has more and much harder technical prerequisites than the diagrammatic proof.
Despite the moralizing about Property $2R$, it does offer the serious benefit of being easy to implement in practice.
Note that using Property $2R$ in this manner required that $D$ was a ribbon disk with a single index one critical point.
Otherwise the Kirby diagram of $-\Sigma$ would have had more than two $2$-handles.
Fortunately, the Property $2R$ approach can be generalized to any ribbon disk.
According to Theorem $5.1$ of \cite{trisect_PropR}, surgery on $\# (n-1) S^1 \times S^2$ that results in $\# n S^1 \times S^2$ must be on a zero framed unknot\footnote{Note that \cite{trisect_PropR} cite it as a special case in expository work of Gordon \cite{gordon_dehn} who ascribes it to Gabai and Scharlemann \cite{scharl_reduce,Gabai_Prop_R}}.
We rephrase this in terms of Property $nR$ as a direct generalization of Lemma \ref{lem:unknot_2R}.
\begin{lem}[Theorem $5.1$ of \cite{trisect_PropR}]
The $n-1$ component unlink $\mathcal{U}_{n-1}$ has Property $nR$.
In particular, let $\mathcal{L} = \mathcal{U}_{n-1} \cup K$ be an $n$ component framed link that contains an $n-1$ component unlink $\mathcal{U}_{n-1}$.
If $\mathcal{L}$ surgers to $\# n S^1 \times S^2$, then there is a sequence of handle slides turning $\mathcal{L}$ into a zero framed $n$ component unlink.
\end{lem}
We can use this to check if the homotopy sphere $\Sigma = E(D) \cup_\phi -X_0(K')$ is standard when $D$ is ribbon.
Draw a Kirby diagram of $S^4 = X_0(K) \cup -E(D)$ with $2$-handle attaching link $\mathcal{L} = K \cup L$ as described in the previous section.
Draw a diagram of $-\Sigma$ as in the proof of Theorem \ref{thm:Z_n_standard}, the attaching link for the $2$-handles in this diagram is $\mathcal{L}' = K' \cup \phi(L)$.
$L$ was originally an unlink and if $\phi(L)$ is still an unlink in this diagram, then $\Sigma$ is standard by Property $nR$ for $\mathcal{U}_{n-1}$.
Somehow this is saying something unsurprising: $L$ represent some aspect of the ribbon structure of $D$ and if $\phi$ leaves it unperturbed, then $\Sigma$ is standard.
\printbibliography

\end{document}